\documentclass[11pt,oneside]{article}
\usepackage{amssymb,amsmath}
\usepackage[english]{babel}
\usepackage{pstricks,pst-node,textcomp}
\renewcommand{\le}{\leqslant}
\renewcommand{\ge}{\geqslant}
\renewcommand{\epsilon}{\varepsilon}
\renewcommand{\phi}{\varphi}

\newcommand{\RR}{\mathbb{R}}
\newcommand{\ZZ}{\mathbb{Z}}
\newcommand{\QQ}{\mathbb{Q}}

\newcommand{\NN}{\mathbb{N}}
\newcommand{\FF}{\mathbb{F}}
\newcommand{\CC}{\mathbb{C}}
\newcommand{\PP}{\mathbb{P}}
\renewcommand{\AA}{\mathbb{A}}

\newtheorem{lemma}{Lemma}
\newtheorem{theorem}{Theorem}
\newtheorem{proposition}{Proposition}

\newtheorem{conja}{Conjecture A}

\newtheorem{conjb}{Conjecture B}

\newtheorem{conjas}{Conjecture A*}

\newenvironment{proof}{\textsc{Proof. }}{\ \newline\hspace*{\fill}$\boxtimes$}

\newcommand{\bigk}{\mathop{\mathbf{K}}}

\begin{document}

\title{Continued fractions of cubic Laurent series}

\author{
 Dzmitry Badziahin
}

\maketitle

\begin{abstract}
    We construct continued fraction expansions for several families
    of the Laurent series in $\QQ[[t^{-1}]]$. To the best of the
    author's knowledge, this is the first result of this kind since
    Gauss derived the continued fraction expansion for $(1+t)^r$, $r\in\QQ$ in
    1813. As an application, we apply an
    analogue of the hypergeometric method to one of those families
    and derive non-trivial efficient lower bounds on the distance $|x -
    \frac{p}{q}|$ between one of the real roots of $3x^3 - 3tx^2-3ax+at$,
    $a,t\in\ZZ$ and any rational number, under relatively mild
    conditions on the parameters
    $a$ and $t$. We also show that every real cubic irrational $x$
    admits a (generalised) continued fraction expansion in a closed
    form that can be explicitly computed. Finally, we provide an
    infinite family of cubic irrationals $x$ that have arbitrarily (but
    finitely) many better-than-expected rational approximations.
    That is, for any $\tau< 3+\frac{15\ln 2}{24}\approx 3.4332...$
    there exists $c>0$ such that
    the inequality $||qx|| < (H(x)^{\tau}
    qe^{c\sqrt{\ln q}})^{-1}$ has many solutions in integer $q$.
\end{abstract}

{\footnotesize{{\em Keywords}: cubic irrationals, continued
fractions, continued fractions of Laurent series, effective rational
approximations of algebraic numbers

Math Subject Classification 2020: 11J68, 11J70}}

\section{Introduction}

Let $\xi\in\RR$ be an algebraic number of degree $d>1$. The
classical theorem of Liouville states that there exists a constant
$c>0$ such that
$$
\left|\xi-\frac{p}{q}\right| >\frac{c}{q^d}
$$
for all rational numbers $p/q$. Moreover, the constant $c=c(\xi)$
can be explicitly computed. Liouville used that result to construct
the first explicit examples of transcendental numbers. Later, in a
series of papers by Thue, Siegel, Dyson, Gelfond and finally
Roth~\cite{roth_1955}, the power of $q$ in Liouville's theorem was
reduced to $2+\epsilon$ for any $\epsilon>0$. This result is sharp
because for $\epsilon=0$, according to the classical Dirichlet
theorem, the opposite inequality holds for infinitely many integers
$p,q$.

While the Roth theorem is very powerful, one of its weaknesses is
that it is ineffective: for any $d\ge 3$ and $\tau>0$ it does not
allow to construct the constants $c>0$ and $q_0$ such that the
inequality
$$
\left|\xi-\frac{p}{q}\right| >\frac{c}{q^{d-\tau}}
$$
holds for all $p/q\in\QQ$ with $q>q_0$.

Effective analogues of the Roth theorem were studied by many number
theorists starting from the 1960s. However, they are still far from
being optimal. One of the approaches is based on Feldman's
refinement of the theory of linear forms in
logarithms~\cite{feldman_1968}. Its advantage is that it improves
the estimate of Liouville for all algebraic numbers. However, this
improvement is usually extremely tiny. For state-of-the-art results
regarding this approach, we refer to the book of
Bugeaud~\cite{bugeaud_2018}.

Another approach was introduced by Bombieri in
1982~\cite{bombieri_1982}. He managed to make the technique of
Dyson~\cite{dyson_1947} from 1947 effective for an infinite class of
number fields of large degree. Later, this technique was further
developed by Bombieri himself, Mueller, Vaaler and Van~der~Poorten
among others. For example, in~\cite{bpv_1996} the authors considered
cubic extensions of number fields. In particular, they show that for
all $\epsilon>0$, the parameters $a,b\in \ZZ$ with $|a|>e^{1000}$
and $|a|> C |b|^{2+\epsilon}$, and for one of the roots $\xi$ of the
cubic equation $x^3 + ax + b=0$, the inequality
$$
\left|\xi - \frac{p}{q}\right| <
q^{-\frac{2\log(|a|^3)}{\log(|a|^3/b^2)} -
\frac{14}{(3\log|a|)^{1/3}}- \epsilon}
$$
has finitely many solutions. Moreover, all of them can be
effectively found. However, huge bounds on the coefficient $a$,
denominator $q$ and the implied constants make this result hardly
applicable in practice. We end this paragraph with the work of
Wakabayashi~\cite{wakabayashi_2002} who provided better effective
bounds on rational approximations for many cubic algebraics,
compared to~\cite{bpv_1996}.


Now we will dwell on the third approach which historically appeared
first and is usually called a hypergeometric method. It was
introduced in 1964 in the work of A.~Baker~\cite{baker_1964} and was
later improved by Chudnovsky~\cite{chudnovsky_1983}. This method
provides much better lower bounds for the distance $|\xi - p/q|$
compared to the previous methods, but only for algebraic numbers of
certain specific forms. The most thoroughly studied numbers are
$\big(1+\frac{a}{n}\big)^r$ where $r\in\QQ$ and $a,n\in \ZZ$ with
$|a|$ considerably smaller than $n$. The most recent results about
these functions are perhaps due to Bennett~\cite{bennett_1997} and
Voutier~\cite{voutier_2007}. For example, Voutier showed that
$$
\left|\sqrt[3]{2}-\frac{p}{q}\right| > \frac{1}{4q^{2.4325}}
$$
for all integer $p$ and $q$. In the last three decades, the
hypergeometric method helped to achieve similar results about some
other families of algebraic numbers, see for example~\cite{lpv_1999,
voutier_2010, wakabayashi_2002}.



The core idea of the hypergeometric method is to consider Pad\'e
approximants of an algebraic Taylor series $x(z)\in \QQ[[z]]$ and
show that their specialisations at $z = \frac{a}{n}$ provide a
family of good rational approximations to the algebraic number
$x(a/n)$. One of the ways to do that for $(1+z)^r$ is to describe
those approximants as the convergents of the following continued
fraction which was already known by Gauss:
$$
(1+z)^r = \frac{1}{\displaystyle 1 + \frac{-rz}{\displaystyle 1+
\frac{(r+1)z}{2 + \frac{(1-r)z}{2 + \cdots}}}}
$$
Here its $(2k+1)$'st partial quotient is $\frac{(k-r)z}{2k+1}$ and
its $2k$'th partial quotient is $\frac{k(r+k)z}{2k}$. To the best of
the author's knowledge and to their surprise, it seems that no
continued fraction expansions are known for any other algebraic
series in $\QQ[[z]]$ of degree at least 3, since the work of Gauss
in 1813.

In this paper, instead of $\QQ[[z]]$, we consider the space
$\QQ[[t^{-1}]]$ of the Laurent series. It admits a very similar
theory of continued fractions to that in the classical case of real
numbers. We make a brief introduction to it in
Section~\ref{ssec_1_2} and refer to~\cite{poorten_1998} for more
details.

Below we present the list of several families of cubic algebraic
Laurent series that admit the continued fraction expansions in
closed form. There is no reason to believe that this list is
exhaustive. For each item there, $x=x(t)$ is the unique Laurent
series in $\QQ[[t^{-1}]]$ that solves the corresponding cubic
equation and satisfies $\deg(x)>0$. The existence and uniqueness of
such $x$ is justified in Section~\ref{sec1} (Lemma~\ref{lem17}). The
partial quotients in continued fractions are overlined to indicate
that the rule for them is periodic, while the partial quotients
themselves are not. The term $k$ in the notation indicates the
periodic template's number where we start counting from $k=0$.

{\bf 1.} For $3x^3 - 3tx^2 - 9x + t=0$ the continued fraction of $x$
is
$$
\bigk\left[ \begin{array}{l}  \\ t \end{array} \overline{\begin{array}{l}(3k+2)(3k+4)\\
(2k+3)t
\end{array}\hspace{-2ex}}\;\; \right]
$$\medskip

{\bf 2.} For $3x^3 - 3tx^2 + 9x - t=0$ the continued fraction of $x$
is
$$
\bigk\left[ \begin{array}{l}  \\ t \end{array} \overline{\begin{array}{l}(3k+2)(3k+4)\\
(-1)^{k+1}(2k+3)t
\end{array}\hspace{-1ex}}\;\; \right]
$$\medskip

{\bf 3.} For $x^3 - tx^2 - at=0$ where $a\in\QQ$ is a parameter, the
continued fraction of $x$ is {\small
$$
\bigk\left[ \begin{array}{l}  \\ t \end{array} \overline{\begin{array}{llll}3(12k+1)(3k+1)a&3(12k+5)(3k+2)a&3(12k+7)(6k+5)a& 3(12k+11)(6k+7)a\\
(8k+3)t&(8k+5)t&2(8k+7)t&(8k+9)t
\end{array}\hspace{-1ex}}\;\; \right]
$$}

{\bf 4.} For $x^3 - tx^2 - a=0$ where $a\in\QQ$ is a parameter, the
continued fraction of $x$ is {\footnotesize
\begin{equation}\label{no4}
\bigk\left[ \begin{array}{l}  \\ \!\!t \end{array} \overline{\begin{array}{llll}3(12k+1)(3k+1)a&3(12k+5)(3k+2)a&3(12k+7)(6k+5)a& 3(12k+11)(6k+7)a\\
(8k+3)t^2&(8k+5)t&2(8k+7)t^2&(8k+9)t\!\!
\end{array}\hspace{-1.5ex}}\;\; \right]
\end{equation}} 

{\bf 5.} For $3x^3 - 3tx^2 - 3ax + at=0$ where $a\in\QQ$ is a
parameter, the continued fraction of $x$ is
\begin{equation}\label{no5}
\bigk\left[ \begin{array}{l}  \\ t \end{array} \overline{\begin{array}{llll}2(3k+1)a&(6k+1)a&2(3k+2)a^2& (6k+5)a^2\\
3(4k+1)t&t&3(4k+3)t(t^2+2a)&t
\end{array}\hspace{-1ex}}\;\; \right]
\end{equation}

{\bf 6.} For $x^3 + (t-2)x^2 - 2(t-2)x + 2(t-2)=0$ the continued
fraction of $x$ is {\small
$$
\bigk\left[ \begin{array}{l}  \\ -t \end{array} \overline{\begin{array}{lll}2(6k+1)(3k+1)&6(4k+1)(3k+2)&3(4k+5)(6k+5)\\
(-1)^{k}((4k+1)t +
2k)&(-1)^{k+1}(4k+3)(t^2+3t-1)&(-1)^{k}((4k+5)t+2k+3)
\end{array}\hspace{-1ex}}\;\; \right]
$$}

{\bf Remark 1.} In the above list, we provide one representative
from the infinite equivalence class of continued fractions. Indeed,
given one continued fraction of some algebraic series, we can
construct continued fractions for many other series by simply
appending or removing partial quotients at the front of the
continued fraction or by replacing $t$ with a polynomial $T\in
\QQ[t]$. We provide a rigorous definition of the equivalence
relation in Section~\ref{ssec_1_2}.

{\bf Remark 2.} The second equation transforms to the first one if
one makes the change of variables $x \mapsto ix$, $t\mapsto it$
where $i$ is the imaginary unit. Hence the first two series will
become equivalent under the extended equivalence relation. On the
other hand, if we specialise both series $x$ for $t\in\QQ$, we
obtain essentially different continued fractions of algebraic
numbers. The situation with series \textnumero 3 and 4 is somewhat
similar. If instead of the constant parameter $a\in\QQ$
in~\textnumero 4 one uses the linear function $at$ and then makes
cancelations of $t$ from numerators and denominators of the
resulting continued fraction, they come up with the continued
fraction~\textnumero 3.

{\bf Remark 3.} One can check that the series $x = (1+t^{-1})^{1/3}$
from the hypergeometric approach described above is equivalent to
the series \textnumero~2. Indeed, they transform one to another by
the change of the variables $t\mapsto \frac{-t-3}{6}$ and $x\mapsto
\frac{1+x}{1-x}$.

In Section~\ref{sec1} we establish the continued fractions for
algebraic series~\textnumero 1 and 5. Given Remark~2, that
automatically establishes the continued fraction for
series~\textnumero 2 as well. The proofs for the other continued
fractions are very similar. Therefore, we only provide the necessary
data for them in Appendix~\ref{app1}. An interested reader can
substitute that data into the proofs from Subsection~\ref{subsec32}
and verify the remaining continued fractions.

The above continued fractions shed the light on rational
approximations to cubic irrationals. First of all, in
Section~\ref{sec5} we show that an analogue of the hypergeometric
method applies to the series~\textnumero 5 and gives a substantial
effective improvement of Liouville's estimate on $\big|\xi -
\frac{p}{q}\big|$ for the solutions of the equation
$3x^3-3tx^2-3ax+at = 0$ under certain relatively mild conditions on
integers $a$ and $t$. 
For convenience, we use the notation $||\xi||$ which means the
distance from $\xi$ to the nearest integer.

\begin{theorem}\label{th3}
Let $a,t$ be positive integers such that $t^2\ge 9a$ and $\xi$ be
the unique real solution of the equation $3x^3 - 3tx^2-3ax+at=0$
that satisfy $\xi>\frac{t}{2}$. Define
$$
\tau_1 := \frac{4(3t^2+a)}{\sqrt{\pi}},\quad \tau_2
:=\frac{105\sqrt{3}e^2a^4}{2\sqrt{\pi}t(t^2+2a)(3t^2+a)},
$$
$$
c_2
 = \frac{144(t^2+a)^3}{c^2_1e^2},\quad c_3
 := \frac{c^2_1e^2t^4(t^2+2a)^2}{9a^6(t^2+a)}\quad \mbox{and}
$$
$$
c_1:=\frac{1}{\sqrt{3}e} \cdot \exp\left(-\sum_{p\in\PP\atop p\ge 5}
\frac{\ln p}{p(p-1)}\right) \approx 0.16948.
$$
Assume that $c_3
>e$. Then for all $q\ge \frac{c_3
}{2\tau_2 }$ one has
\begin{equation}\label{eq2_prop4}
||q\xi|| > \frac{(\log c_3 )^{1/2}}{6\tau_1 c^2_2 (2\tau_2
)^\frac{\log c_2 }{\log c_3 }} \cdot q^{-\frac{\log c_2 }{\log c_3
}} \cdot \log (2\tau_2 q)^{-\frac{\log c_2 }{\log c_3 }-\frac12}.
\end{equation}
\end{theorem}

While the parameters in the theorem look complicated, it gives good
non-trivial lower bounds on $||q\xi||$. They become better than in
the classical Liouville theorem if $\log c_2
 < 2\log c_3
 $ or
equivalently, $c_2
 <c_3
 ^2$. In terms of $a$ and $t$, the last
inequality is
$$
11664 a^{12}(t^2+a)^5 < c_1^6e^6 t^8 (t^2+2a)^4.
$$
With software like Wolfram Mathematica, one can verify that this
inequality is always satisfied as soon as $t>10.34a^2$. To the best
of the author's knowledge, Theorem~1 is the first result of this
type for the family of cubic irrationals with the minimal polynomial
$3x^3-3tx^2-3ax+at$.

Also notice that as the parameter $a$ is fixed and $t\to\infty$, the
fraction $\frac{\log c_2 }{\log c_3 }$ approaches 1 and hence for
all $\epsilon>0$ there exists $t_0(\epsilon)$ such that for all
$t>t_0$ the solution $\xi$ from Theorem~\ref{th3} satisfies
$||q\xi||\gg q^{-1-\epsilon}$. The computations of the lower bounds
on $||q\xi||$ for some values of $t$ and $a$ are provided in
Table~\ref{tab1}.

\begin{table}[ht]
\renewcommand{\arraystretch}{1.2}
\noindent\begin{tabular}{|c|c|c|c|c|c|} \hline
$a$&$t$&Equation& \!For all\! $q\ge$& one has $||q\xi||\ge$& \!\!\footnotesize Numeric evidence\!\!\\
\hline
1&3&\footnotesize$x^3-3x^2-x+1=0$& 29 & \footnotesize $\gg q^{-4.276}$ & $\gg q^{-1.57}$\\
\hline 1&4&\footnotesize$3x^3-12x^2-3x+4=0$& 536 & \footnotesize
$\gg
q^{-3.166}$& $\gg q^{-1.98}$\\
\hline
1&11&\footnotesize$3x^3-33x^2-3x+11=0$& 27812480 & \footnotesize$1.4\cdot 10^{-16}q^{-1.963}\log^{-2.463}(0.0016 q)$& $\gg q^{-1.26}$\\
\hline
1&12&\footnotesize$x^3-12x^2-x+4=0$& 71929013 & \footnotesize$7.5\cdot 10^{-17}q^{-1.92}\log^{-2.42} (10^{-3}q)$& $\gg q^{-1.49}$\\
\hline
1&30&\footnotesize$x^3-30x^2-x+10=0$& $1.67\cdot 10^{12}$&\footnotesize \hspace{-1.5ex} $5.14\cdot 10^{-20}q^{-1.62}\log^{-2.12} (1.04\cdot10^{-5}q)$\!\!& $\gg q^{-1.35}$\\
\hline
2&6&\footnotesize$x^3-6x^2-2x+4=0$& 49 & \footnotesize$\gg q^{-5.81}$& $\gg q^{-1.93}$\\
\hline
2&42&\footnotesize$x^3-42x^2-2x+28=0$&$6.6\cdot 10^{10}$& \footnotesize$3.7\cdot 10^{-21}q^{-1.994}\log^{-2.494} (3.1\cdot 10^{-5}q)$&$ \gg q^{-1.35}$\\
\hline
2&43&\footnotesize$3x^3-126x^2-6x+86=0$&$8.6\cdot 10^{10}$& \footnotesize$3\cdot 10^{-21}q^{-1.984}\log^{-2.744} (2.7\cdot 10^{-5}q)$& $\gg q^{-1.63}$\\
\hline
3&7&\footnotesize$x^3-7x^2-3x+7=0$& 5 & \footnotesize$\gg q^{-12.19}$& $\gg q^{-1.46}$\\
\hline
3&94&\footnotesize$x^3-94x^2-3x+94=0$&$8\cdot 10^{12}$& \footnotesize$6.4\cdot 10^{-24}q^{-1.997}\log^{-2.497} (2.7\cdot 10^{-6}q)$& $\gg q^{-1.29}$\\
\hline
3&95&\footnotesize$x^3-95x^2-3x+95=0$&$9\cdot 10^{12}$& \footnotesize$5.8\cdot 10^{-24}q^{-1.993}\log^{-2.493} (2.6\cdot 10^{-6}q)$& $\gg q^{-1.15}$\\
\hline
\end{tabular}
\caption{lower bounds for $||q\xi||$}\label{tab1}
\end{table}

It is worth mentioning that the lower bounds in Theorem~\ref{th3}
are not best possible. Direct numerical computations of the first
several hundred convergents $p_n/q_n$ of the continued
fraction~\eqref{no5} suggest that better bounds should take place.
For example, for $a=1$ they indicate that even with $t=3$ we should
get a non-trivial lower bound on $||q\xi||$. We provide those better
numerical bounds in the last column of Table~\ref{tab1}. They can be
considered as the limitation of our method. In fact, for some pairs
$(a, t)$ the numerators and denominators of the convergents seem to
be both divisible by a big power of two or three, which leads to
substantially better estimates on $\gcd(p_{4k+2}, q_{4k+2})$ than in
Lemma~\ref{lem5}. But even without this phenomenon, some of the
estimates in the proof of Theorem~\ref{th3} do not look optimal (for
example, one in~\eqref{lem5_eq} and
probably~\eqref{eq11},~\eqref{eq10}).

We expect that an analogue of the hypergeometric method also applies
to other cubic series from the list. We will explore that direction
in further research. It is worth mentioning that every cubic series
can be easily transformed to the series~\textnumero 4 by an
appropriate M\"obius map.

Another important consequence of the constructed continued fractions
is that they provide cubic irrationals $\xi$ with very good rational
approximations $p/q$. In other words, with help of these continued
fractions one can construct infinitely many pairs $(p/q,
\xi)\in\QQ\times \AA_3$ with very small distances $|p/q - \xi|$ in
terms of the heights of $p/q$ and $\xi$. Here, by $\AA_3$ we denote
the set of all real cubic irrationals. As numerical computations
suggest (we will talk about them in detail later), these distances
are close to, if not one of, smallest possible among all pairs in
$\QQ\times \AA_3$. This allows us to investigate the question of
providing an effective uniform lower bound for $||q\xi||$ that is
satisfied for all cubic irrationals~$\xi$.

Denote by $P_\xi(x)$ the minimal polynomial of $\xi$ with integer
coprime coefficients. In this paper, by the height of $\xi$ we mean
its naive height, i.e. the maximal absolute value of the
coefficients of $P_\xi$, and denote it by $H(\xi)$. In
Section~\ref{sec4} we prove the following result.

\begin{theorem}\label{th4}
For any $\tau < 3 + \frac{15\ln 2}{24}\approx 3.4332...$ there
exists an effectively computable constant $c>0$ such that for all
$N_0\in \NN$ there exist infinitely many cubic irrationals $\xi\in
K_3$ such that the inequality
\begin{equation}\label{eq_th3}
||q\xi|| < \frac{1}{H(\xi)^\tau q e^{c\sqrt{\ln q}}}
\end{equation}
has more than $N_0$ solutions $q\in \NN$.
\end{theorem}

Notice that, according to the Khintchine theorem, for almost all
$\xi\in\RR$ and all but finitely many $q\in\ZZ$ one must have
$||q\xi||\ge (q \ln^{1+\epsilon} q)^{-1}$. Therefore
Theorem~\ref{th4} provides cubic irrationals with much better
estimates on $||q\xi||$ than what is expected, for many (but
finitely many) values of $q$.

We also provide heuristic evidence that the power $\tau$ in
Theorem~\ref{th4} is nearly best possible and formulate the
following conjecture.

\begin{conja}
For any $\epsilon>0$ define $\tau_0 = 3 + \frac23 \ln2 +\epsilon
\approx 3.462+\epsilon$. There exists an absolute constant $C =
C(\epsilon)$ such that for all real cubic irrationals $\xi$ all
partial quotients $a_n$ satisfy
$$
|a_n|\le C n^{1+\epsilon} H(\xi)^{\tau_0}.
$$
\end{conja}

In view of the classical estimates for the convergents: $q_n \le
\phi^{n-1}$ and $||q_n\xi||< (a_{n+1}q_n)^{-1}$ where $\phi =
\frac{\sqrt{5}+1}{2}$, we can reformulate Conjecture~A in a more
standard form:

\begin{conjas}\label{conj2}
Let $\epsilon>0$, $\tau_0$ and $C$ be as in Conjecture~A. There
exists an effectively computable constant $c$ such that for all real
cubic irrational numbers $\xi$ and all $q\in \NN$,
\begin{equation}\label{conja_eq2}
||q\xi|| \ge \frac{c}{H(\xi)^{\tau_0} q(\ln q)^{1+\epsilon}}.
\end{equation}
Moreover, one can take $c = \frac{\ln^2\phi}{C}$.
\end{conjas}

This conjecture not only implies the Roth theorem for cubic
irrationals but also the stronger conjecture of
Lang~\cite{lang_1991} which says that $||q\xi||\gg (q\ln
q^{1+\epsilon})^{-1}$ where the implied constant may depend on
$\xi$. Conjecture~A is out of reach by current methods and we
believe it will be extremely hard to prove.

In Section~\ref{sec4}, we provide heuristic arguments supporting the
statement that all partial quotients $a_n$ which come from the
discovered continued fractions are not bigger than
\begin{equation}\label{conja_eq}
|a_n|\le Cn^2 H(\xi)^{3+\frac{2\ln 2}{2.88}},
\end{equation}
where the constant $C$ can be explicitly computed. To further verify
Conjecture~A, we have numerically computed the continued fraction
expansions of a wide set of algebraic numbers. For the sake of
simplifying the code, we only considered numbers inside the interval
$(0,1)$ and if a cubic equation has more than one root in that
region, we considered only one of them. Therefore our search is not
exhaustive. For all polynomials of height at most 2, we computed the
partial quotients up to index 10000, for all $P$ with $H(P)\le 5$ up
to index 5000, for all $P$ with $H(P)\le 10$ up to index 1000 and
for all $P$ with $H(P)\le 100$ up to index 50. As a result, we have
found only 6 instances when the constant $C$ in~\eqref{conja_eq} is
bigger than 8. All of them are for $\xi$ with $H(\xi)\le 7$. On top
of that, there are only four more instances with $C>2$. The largest
discovered $C=17.751\ldots$ is for the root of $2x^3+2x^2+2x-1=0$.
Since the search was not exhaustive, we can not claim this is the
right value for Conjecture~A. However it is probably safe to write
that the conjecture is satisfied for $C=100$ and that for $C=2$ it
is possible to explicitly write down all its counterexamples. In
Appendix~\ref{app2}, we list all discovered cubic irrationals with
$C>2$. We conclude with a conjecture, where all the parameters are
made specific.

\begin{conjb}\label{conj1}
Let $\tau_3 = 3+ \frac{2\ln 2}{2.88} \approx 3.4814$. For all real
cubic irrationals $\xi$ all partial quotients $a_n$ satisfy
$$
|a_n| \le 100 n^2 H(x)^{\tau_3}.
$$
Also, for all $q\in \ZZ$ and all cubic irrationals $\xi$ one has
$$
||q\xi|| \ge \frac{1}{440 H(\xi)^{\tau_3} q(\ln q)^2}.
$$
\end{conjb}

As the last application, in Section~\ref{sec6} we demonstrate that
every cubic irrational number admits a (generalised) continued
fraction expansion in a closed form.


\begin{theorem}\label{th2}
Let $\xi\in\RR$ be a cubic algebraic number. Then there exists a
M\"obius transformation $\mu: x \mapsto \frac{ax+b}{cx+d}$ with
$a,b,c,d\in\ZZ$ such that $\mu(\xi)$ enjoys the continued fraction
expansion of the form~\eqref{no4} which converges to $\mu(\xi)$.
\end{theorem}

Notice that if we have a continued fraction expansion $\bigk \left[
\begin{array}{ccc}
&\beta_1&\beta_2\\
a_0&a_1&a_2
\end{array} \cdots
\right]$ of $y = \mu(x)$ then we can construct a continued fraction
for $x = \mu^{-1}(y) = \frac{uy+v}{wy+z}$ too as follows:
$$
x = \frac{u}{w} + \frac{vw-uz}{\displaystyle wz+w^2a_0 +
\frac{w^2\beta_1}{\displaystyle a_1 + \frac{\beta_2}{a_2 +
\cdots}}}.
$$
Hence we can provide a continued fraction expansion in the closed
form for all cubic $\xi\in\RR$.

\section{Laurent series}\label{ssec_1_2}

Let $\FF$ be a field. Consider the set $\FF[[t^{-1}]]$ of the
Laurent series together with the valuation: $||\sum_{k=-d}^\infty
c_kt^{-k}|| = d$, the biggest degree $d$ of $t$ having non-zero
coefficient $c_{-d}$. Sometimes in the paper we call it the degree
of a series because it matches the definition of the degree for
polynomials. We also use the notation $[f]$ for the polynomial part
of $f$, i.e. $[f]:=\sum_{k=-d}^0 c_kt^{-k}$. It is well known that
in this setting the notion of continued fraction is well defined. In
other words, every $f(t)\in \FF[[t^{-1}]]$ can be written as
$$
f(t) = [a_0(t), a_1(t),a_2(t),\ldots] = a_0(t) + \frac{1}{a_1(t) +
\frac{1}{a_2(t) + \cdots}} ,
$$
where $a_i(t)\in \FF[t]$ are called partial quotients, and $\deg(a_i)\ge 1$
for all $i\ge 1$. We refer the reader to a nice survey~\cite{poorten_1998}
for more properties of the continued fractions of Laurent series.

In this paper, we have $\FF = \QQ$ and then for a given $f\in\QQ[[t^{-1}]]$
the partial quotients $a_i$ are polynomials with rational coefficients. It
will be more convenient to renormalise this continued fraction by multiplying
its numerators and denominators by appropriate integer numbers so that all
the coefficients of $a_i$ are integer:
$$
 f(t) = \frac{1}{\beta_0}\left(a_0(t) + \frac{\beta_1}{a_1(t) + \frac{\beta_2}{a_2(t) + \frac{\beta_3}{a_3(t) + \cdots}}}\right) =: \bigk \left[ \begin{array}{ccc}
\beta_0&\beta_1&\beta_2\\
a_0(t)&a_1(t)&a_2(t)
 \end{array}\cdots\right],
$$
where $\beta_i\in\ZZ\setminus\{0\}$ and $a_i(z)\in \ZZ[t]$ for $i\ge 1$. If
$\beta_0 = 1$ then sometimes we omit $\beta_0$ in the notation and write
$\bigk \left[
\begin{array}{ccc}
&\beta_1&\beta_2\\
a_0&a_1&a_2
\end{array} \cdots
\right]$.

By analogy with the classical continued fractions over $\RR$, by $k$'th
convergent of $f$ we denote the rational function
$$
\frac{p_k(t)}{q_k(t)}:= \bigk \left[ \begin{array}{ccccc}
\beta_0&\beta_1&\beta_2&\cdots&\beta_k \\
a_0(t)&a_1(t)&a_2(t)&\cdots&a_k(t)
 \end{array}\right].
$$
The convergents satisfy the following recurrent relation:
\begin{equation}\label{convp}
p_0(z) = a_0(z), \quad p_1(z) = a_0(z)a_1(z) + \beta_1,\quad p_n(z)
= a_n(z)p_{n-1}(z) + \beta_np_{n-2}(z),
\end{equation}
\begin{equation}\label{convq}
q_0(z) = 1, \quad q_1(z) = a_1(z),\quad q_n(z) = a_n(z)q_{n-1}(z) +
\beta_nq_{n-2}(z).
\end{equation}
The important property of convergents is that they are the best
rational approximants of the series $f$. That is, the analogue of
the Lagrange's theorem is true: $p(t)/q(t)\in \QQ(t)$ with coprime
polynomials $p$ and $q$ is a convergent of $f\in \QQ[[t^{-1}]]$ if
and only if $\deg (f - p/q) < -2\deg(q)$.

By $k$'th full quotient of $f$ we denote the continued fraction
$$
f_k(t):= \bigk \left[ \begin{array}{ccc}
\beta_k&\beta_{k+1}&\beta_{k+2}\\
a_k(t)&a_{k+1}(t)&a_{k+2}(t)
 \end{array}\cdots\right].
 $$
One can easily verify that the consecutive full quotients of $f$ satisfy
the relation:
\begin{equation}\label{def_phi}
f_{k+1} = \phi_{a_k, \beta_k}\circ f_k := \frac{1}{\beta_k f_k -
a_k}.
\end{equation}

In this paper we consider algebraic Laurent series $x(t)$, i.e. solutions of
the equations of the form
\begin{equation}\label{def_x}
b_d x^d + b_{d-1}x^{d-1} + \cdots + b_1x + b_0=0
\end{equation}
where $b_0, b_1,\ldots, b_d\in\QQ[t]$. By a modification of the
Newton-Puiseux theorem, there exists a neighbourhood $U_\infty\subset\CC$ of
infinity such that $x$, considered as a Laurent series in $\CC[[t^{-1}]]$,
converges in $U_\infty \setminus\{\infty\}$ and the limit is a holomorphic
function that we also write as $x(t)$. If $\deg x>0$ then it has a pole at
infinity.

Consider the sequence $p_k/q_k$ of convergents of $x$. Since they are all in
$\CC(t)$, they are all holomorphic in the whole space $\CC$. Suppose that, as
holomorphic functions, $p_k/q_k$ converge uniformly to $f(t)$ in some
neighbourhood $V_\infty$ of $\infty$. Then $f$ is also a holomorphic function
in $V_\infty$ and hence has a Laurent series expansion that is convergent in
some neighbourhood $V^*_\infty \subset V_\infty$ of infinity. But then it
must coincide with the expansion of $x$.

Finally, we can analytically continue both functions $f$ and $x$ and
derive the following statement: as soon as the functions $p_k/q_k$
converge uniformly in some connected neighbourhood $U$ of $\infty$,
the limit $\lim_{k\to\infty} \frac{p_k(t)}{q_k(t)}$ equals $x(t)$,
which is a solution of the equation~\eqref{def_x}. Note that we do
not need the Laurent series of $x$ to converge at $t$ for this
statement to take place. In other words, as soon as we compute the
continued fraction expansion of $x\in\QQ[[t^{-1}]]$, for all $t\in
U$ its specialisation at $t$ is also a continued fraction expansion
of an algebraic number $x(t)$.

If a continued fraction for some algebraic $x\in\QQ[[t^{-1}]]$ is constructed
then one can immediately construct many more continued fractions by one of
the following operations:\\[-5ex]
\begin{itemize}
\item append or remove several first partial quotients;\\[-5ex]
\item replace the variable $t$ by $P(t)$ where $P\in\QQ[t]$.\\[-5ex]
\end{itemize}
Therefore, in order to consider only essentially different continued
fractions, we define the equivalence relation between them. We say
that two series $x$ and $y$ are equivalent if the continued fraction
of one of them can be achieved from that of another one by a finite
number of transformations: adding or removing one partial quotient
at the beginning of the continued fraction and replacing the
variable $t$ by $P(t)$.

One observes that if $x = \bigk \left[
\begin{array}{ccc}
&\beta_1&\beta_2\\
a_0&a_1&a_2
\end{array} \cdots
\right]$ then the continued fraction of $ux + v$ for $u,v\in \ZZ$
can be written as $\bigk \left[
\begin{array}{ccc}
&u\beta_1&\beta_2\\
ua_0 + v&a_1&a_2
\end{array} \cdots
\right]$, i.e. $x\sim ux+v$. It is also easy to verify that $x\sim
1/x$. Indeed, that is obvious if $\deg x\neq 0$, i.e. $a_0(t) = 0$
or $\deg a_0 \ge 1$. In the case when $a_0 = \mathrm{const}$, one
verifies that
$$
\frac{1}{x} = \bigk \left[
\begin{array}{cccc}
&-\beta_1& a_0\beta_2&\beta_3\\
1/a_0&a_0a_1+\beta_2&a_2&a_3
\end{array} \cdots
\right].
$$
By combining the last two statements, we derive that if $x$ and $y$
are related by the M\"obius transformation $y = \frac{ax+b}{cx+d}$,
$a,b,c,d\in\ZZ$, $ad-bc\neq 0$, then $x\sim y$. The other
easy-to-spot condition for equivalent series is as follows: if the
algebraic equation for $x$ is $F(x,t)=0$ and the one for $y$ is
$F(y,P(t))=0$ for some $P\in\ZZ[t]$ then $x\sim y$.

We finish this section with a well known result about
transformations of a continued fraction that do not change its
limit. Since we did not find a good reference for it, we add its
proof here.

\begin{lemma}\label{lem8}
Let $\bigk \left[\cdots
\begin{array}{ccc}
\beta_{i-1}&\beta_i&\beta_{i+1}\\
a_{i-1}&a_i&a_{i+1}
\end{array} \cdots
\right] $ be a continued fraction and $A\neq 0$ a rational number. Then the
transformations
\begin{enumerate}
\item $\beta_{i-1} \mapsto \beta_{i-1}/A, a_{i-1}\mapsto a_{i-1}/A$,
    $\beta_i\mapsto \beta_i/A$
\item $\beta_{i-1} \mapsto \beta_{i-1}/A, a_{i-1}\mapsto a_{i-1}/A$,
    $a_i\mapsto a_i A$, $\beta_{i+1} \mapsto \beta_{i+1}A$
\end{enumerate}
do not change the limit of the continued fraction. Moreover, the convergents
$p_i/q_i$ do not change under these transformations too.
\end{lemma}

\proof Let $(p_n/q_n)_{n\in \ZZ_{\ge 0}}$ be the convergents of the initial
continued fraction and $(p^*_n/q^*_n)_{n\in \ZZ_{\ge 0}}$ the convergents of
the modified one. Then we have $$p^*_{i-1} = \frac{a_{i-1}}{A}p_{i-2}+
\frac{\beta_{i-1}}{A}p_{i-3} = \frac{p_{i-1}}{A}.$$ Analogously, $q^*_{i-1} =
\frac{q_{i-1}}{A}$, therefore $p_{i-1}/q_{i-1} = p^*_{i-1}/q^*_{i-1}$.

Next, in the first case we have $$p^*_i = a_i p^*_{i-1} + \frac{\beta_i}{A}
p_{i-2} = \frac{p_i}{A}.$$ Analogously, we have that $q^*_i = \frac{q_i}{A}$
and again $p^*_i/q^*_i = p_i/q_i$. Then one can easily check that for all
$j>i$, $p^*_j = \frac{p_j}{A}$ and $q^*_j = \frac{q_j}{A}$ and the claim of
the lemma is verified.

In the second case, we have $p^*_i = a_iAp^*_{i-1} + \beta_i p_{i-2} = p_i$
and analogously $q^*_i = q_i$. Also, $p^*_{i+1} = a_{i+1}p^*_i + \beta_{i+1}A
p^*_{i-1} = p_{i+1}$ and the same is true for $q^*_{i+1} = q_{i+1}$. Finally,
since all other partial quotients remain unchanged, we have for all $j>i+1$,
$p^*_j = p_j$ and $q^*_j = q_j$ and the claim is verified. \endproof

\section{Constructing continued fractions}\label{sec1}

\subsection{Riccati equation}

Let $x(t)\in \QQ[[t^{-1}]]$ be a cubic series with the minimal
polynomial $P(x) = b_0 + b_1x + b_2x^2 + b_3x^3\in \QQ[t][x]$. One
of the main tools for constructing the continued fractions of
algebraic series is to understand how the equation for $x$ changes
under the transform $\phi_{a,\beta}$ from~\eqref{def_phi}, where
$a\in \QQ[t]$, $\beta\in \QQ$. While it is not difficult to compute
the minimal polynomial of $\phi_{a,\beta}(x)$ for a given algebraic
$x$, after applying a numbers of such transforms, very quickly the
coefficients of the minimal polynomials of the full quotients
$x_k(t)$ become very complicated. Their degrees usually grow to
infinity and their heights grow rapidly too. It is very hard to
recognise any pattern in them. Instead, we first notice that $x(t)$
satisfies a Riccati differential equation and then verify that the
new equation behaves much nicer under $\phi_{a,\beta}$. A similar
idea was considered by Osgood~\cite{osgood_73}.

Differentiation of $P(x)$ gives
$$
(b_1 + 2b_2x + 3b_3x^2)x' + (b_0' + b_1'x + b_2'x^2 + b_3'x^3)=0.
$$
Therefore, $x'(t)$ belongs to the same cubic extension of $\QQ(t)$
as $x$, and we can decompose it in the basis $1, x, x^2$. One can do
that in the following way. Let $x'$ satisfy the equation $x' = u +
vx + wx^2$ where $u,v,w\in \QQ(t)$ are rational functions to be
found. Then we write
$$
(b_1 + 2b_2x + 3b_3x^2)(u + vx+wx^2) + (b_0' + b_1'x + b_2'x^2 + b_3'x^3)= (\beta + \alpha x)(b_0 + b_1x + b_2x^2 + b_3x^3).
$$
Comparison of terms at $x^4$ and $x^3$ gives
$$\alpha = 3w\; \mbox{ and }\; \beta = 3v + \frac{b_3' - b_2w}{b_3}.
$$
Next, comparison of terms at $1, x$ and $x^2$ gives the system of
linear equations
$$
\left( \begin{array}{ccc}
b_3b_1&-3b_3b_0&b_2b_0\\
2b_3b_2&-2b_3b_1&b_2b_1-3b_3b_0\\
3b_3^2&-b_3b_2&b_2^2 - 2b_3b_1
\end{array}\right) \left( \begin{array}{c}
u\\v\\w
\end{array}\right) = \left( \begin{array}{c}
b_0b_3' - b_3b_0'\\
b_1b_3' - b_3b_1'\\
b_2b_3' - b_3b_2'
\end{array}\right).
$$
The application of Cramer's rule to this equation leads to the
following proposition:
\begin{proposition}\label{prop1}
Let $x(t)\in \QQ((t^{-1}))$ be a solution of a cubic equation
$b_3x^3 + b_2x^2 + b_1x + b_0=0$. Then $x$ is a solution of the
following Riccati differential equation:
\begin{equation}\label{riccati}
Dx' = A + Bx + Cx^2
\end{equation}
where $A, B, C, D\in \QQ[t]$ are computed by the following formulae:
\begin{equation}\label{eq_ricd}
D = \det\left( \begin{array}{ccc}
b_1&-3b_0&b_2b_0\\
2b_2&-2b_1&b_2b_1-3b_3b_0\\
3b_3&-b_2&b_2^2 - 2b_3b_1
\end{array}\right) = 4b_3b_1^3 + 4b_2^3b_0+27b_3^2b_0^2 - b_2^2b_1^2 - 18b_3b_2b_1b_0;
\end{equation}
\begin{equation}\label{eq_rica}
A = \frac{1}{b_3}\det\left( \begin{array}{ccc}
b_0b_3'-b_3b_0'&-3b_0&b_2b_0\\
b_1b_3'-b_3b_1'&-2b_1&b_2b_1-3b_3b_0\\
b_2b_3'-b_3b_2'&-b_2&b_2^2 - 2b_3b_1
\end{array}\right);
\end{equation}
\begin{equation}\label{eq_ricb}
B = \frac{1}{b_3}\det\left( \begin{array}{ccc}
b_1&b_0b_3'-b_3b_0'&b_2b_0\\
2b_2&b_1b_3'-b_3b_1'&b_2b_1-3b_3b_0\\
3b_3&b_2b_3'-b_3b_2'&b_2^2 - 2b_3b_1
\end{array}\right);
\end{equation}
\begin{equation}\label{eq_ricc}
C = \det\left( \begin{array}{ccc}
b_1&-3b_0&b_0b_3'-b_3b_0'\\
2b_2&-2b_1&b_1b_3'-b_3b_1'\\
3b_3&-b_2&b_2b_3'-b_3b_2'
\end{array}\right).
\end{equation}
\end{proposition}

{\bf Remark.} Notice that $-D$ equals the discriminant of $P(x)$.

The next step is to investigate how this equation changes under the
transform $\phi_{a,\beta}$.

\begin{proposition}\label{prop2}
Let $x\in\QQ((t^{-1}))$ be the solution of the Riccati equation $Dx' = A+Bx+Cx^2$. Then $y = \phi_{a,\beta}(x)$ satisfies the equation $\tilde{D} y' = \tilde{A} + \tilde{B}y + \tilde{C}y^2$ where
$$
\left(\begin{array}{c}
\tilde{A}\\\tilde{B}\\\tilde{C}\\\tilde{D}
\end{array}\right) = \left(\begin{array}{cccc}
0&0&-1&0\\
0&-\beta&-2a&0\\
-\beta^2&-\beta a&-a^2&\beta a'\\
0&0&0&\beta
\end{array}\right)\left(\begin{array}{c}
A\\B\\C\\D
\end{array}\right)
$$
\end{proposition}

\proof Notice that
$$
\beta Dy' = \beta D\frac{d}{dt} \left(\frac{1}{\beta x-a}\right) = \frac{-\beta D(\beta x' - a')}{(\beta x-a)^2} = \frac{-\beta^2(A + Bx + Cx^2) + \beta Da'}{(\beta x-a)^2}.
$$
On the other hand,
$$
A' + B'y + C'y^2 = \frac{A'(\beta x-a)^2 + B' (\beta x-a) + C'}{(\beta x-a)^2}.
$$
Equating the coefficients for $1,x$ and $x^2$ of the numerators of
both expressions gives
$$
A' = -C,\quad B' = -\beta B - 2 a C, \quad C' = -\beta^2 A - \beta aB - a^2C + \beta a' D.
$$
Then the Proposition follows immediately. \endproof

We have shown that every solution of a cubic equation is also a
solution of the corresponding Riccati equation. However, each cubic
equation may have more than one solution $x\in \CC[[t^{-1}]]$. In
fact, by the Newton-Puiseux theorem it always has three solutions,
counting multiplicities, in the extension $\CC[[t^{-1/6}]]$. In
further discussion, we will always assume that all the solutions are
distinct, otherwise the cubic equation can be reduced to one of a
smaller degree. Riccati equations also have multiple solutions. The
aim of the next lemma is to link a given solution of the cubic
equation with a solution of the corresponding Riccati one.

\begin{lemma}\label{lem4}
Let $x_1,x_2,x_3\in\CC[[t^{-1/6}]]$ be three distinct solutions of
the cubic equation $b_0 + b_1x + b_2x^2 + b_3x^3=0$ where
$b_0,b_1,b_2,b_3\in\CC[t]$. Suppose that $\deg x_1\ge 1$ and $\deg
x_2, \deg x_3\le 0$. Then $x_1$ is the only solution of the
corresponding Riccati equation~\eqref{riccati} with the property
$\deg x\ge 1$.
\end{lemma}

\proof Without loss of generality, assume that $\deg x_2\ge \deg
x_3$. Proposition~\ref{prop1} implies that all three series
$x_1,x_2,x_3$ are the solutions of the Ricatti
equation~\eqref{riccati}. Let $x\in\CC[[t^{-1/6}]]$ be any other
solution of~\eqref{riccati}, distinct from the three solutions
above. Then it is well known that
$$
K := \frac{(x_1-x_3)(x_2-x)}{(x_2-x_3)(x_1-x)}
$$
is a constant. Then we have
$$
x = \frac{x_2(x_1-x_3) - Kx_1(x_2-x_3)}{(x_1-x_3)-K(x_2-x_3)}
$$
The degree of the right hand side is at most $\deg x_2 \le 0$.
\endproof

\begin{lemma}\label{lem3}
Let $a\in \QQ[t]$ and $\beta\in\QQ$ such that $\deg a\ge 1$. If the
Riccati equation~\eqref{riccati} has the unique solution
$x\in\QQ((t^{-1}))$ such that $\deg x\ge 1$ then the Riccati
equation $\tilde D y' = \tilde A + \tilde By + \tilde Cy^2$ for $y =
\phi_{a, \beta}(x)$ has at most one solution $y$ with $\deg y\ge 1$.
\end{lemma}

\proof One can verify that $\phi_{a,\beta}$ is the bijection between
the solutions of~\eqref{riccati} and $\tilde D y' = \tilde A +
\tilde By + \tilde Cy^2$. Now, if $y_1,y_2$ are both the solutions
of the latter equation with $\deg y_1, \deg y_2\ge 1$ then $\deg
\phi^{-1}_{a,\beta}(y_i) =
\deg\big(\big(\frac{1}{y_i}+a\big)/\beta\big) \ge 1$ for both
$i=1,2$. By assumption, such a series must be unique, i.e.
$\phi_{a,\beta}^{-1} (y_1) = \phi_{a,\beta}^{-1}(y_2)$ or
equivalently $y_1=y_2$.
\endproof

After we verify that the equation~\eqref{riccati} has the unique
solution of degree at least 1, one can readily construct it by
comparing the coefficients at the corresponding powers of $t$
in~\eqref{riccati}.  In the next two lemmata we will cover two cases
that will be needed later. However an analogous approach can be
applied for any quadruple of polynomials $A, B, C$ and $D$.

\begin{lemma}\label{lem1}
Let $A, B, C, D\in \QQ[t]$ be polynomials of degrees $s_a, s_b, s_c$
and $s_d$ respectively such that $A = \sum_{n=0}^\infty r_{a,n}t^n$,
$B = \sum_{n=0}^\infty r_{b,n}t^n$, $C = \sum_{n=0}^\infty
r_{c,n}t^n$, $D = \sum_{n=0}^\infty r_{d,n}t^n$. Here, by
convention, $r_{a,i} = r_{b,i} = r_{c,i} = r_{d,i} = 0$ for all
$i>s_a, i>s_b, i>s_c$ and $i>s_d$ respectively. Assume that $s_d =
s_b + 1$, $s_b > s_c$, $s_b\ge s_a$ and $d r_{d, s_d} \neq r_{b,
s_b}$ for all $d\in\NN$ with $0<d<2(s_b - s_c)$. Let $x$ be a
solution of the equation~\eqref{riccati} of degree at least one.
Then its degree and leading coefficient are computed by the formulae
\begin{equation}\label{eq_lem1}
s_x = \deg x = s_b - s_c, \quad r_{x,s_x} = \frac{r_{d,s_d}(s_b -
s_c) - r_{b,s_b}}{r_{c,s_c}}.
\end{equation}
\end{lemma}

\proof Let $s_x$ be the degree of $x$, i.e. $x = \sum_{n=-s}^\infty
r_{x, -n}t^{-n}$ with $r_{x,s}\neq 0$. Then the degrees of the terms
$Dx', A, Bx$ and $Cx^2$ are $s_d + s_x-1$, $s_a$, $s_b + s_x$ and
$s_c+2s_x$ respectively. In order for $r_{x,s}$ to be nonzero, the
maximum of those four values need to be achieved at at least two of
them. For $s_x>s_b-s_c$ we have that
$$
s_c + 2s_x > s_d + s_x-1 = s_b + s_x > s_a
$$
which is a contradiction. Therefore $s_x\le s_b-s_c$.

Suppose now that $0<s_x<s_b - s_c$. Then the comparison of the
coefficients at $t^{s+s_b}$ in~\eqref{riccati} gives $s r_{d,s_d}
r_{x,s} = r_{b,s_b} r_{x,s}$. But that contradicts the conditions of
the lemma and the fact $r_{x,s}\neq 0$.

Hence we have $s_x = s_b - s_c$. Comparison of the leading
coefficients now gives $s_xr_{d, s_d} r_{x,s_x} = r_{b,s_b}
r_{x,s_x} + r_{c,s_c} r_{x,s_x}^2$ and~\eqref{eq_lem1} immediately
follows.

\endproof

Note that we can continue the comparison of the coefficients at $t^{s+s_b}$
in~\eqref{riccati} for $s = s_x-1, s_x-2, \ldots$. That will allow us to
identify all the coefficients of the polynomial part of $x$. However, this
process is quite technical and not very enlightening. Therefore we omit it in
this paper.

\begin{lemma}\label{lem2}
With the same notation as in Lemma~\ref{lem1}, assume that
$s_d>s_b+1$, $s_d>s_c+1$ and $s_d>s_a$. Assume that the
equation~\eqref{riccati} has the unique solution $x$ of degree at
least 1. then its degree and leading coefficient are computed by the
formulae
\begin{equation}\label{eq_lem2}
s_x = s_d-s_c-1,\quad r_{x,s_x} =
\frac{(s_d-s_c-1)r_{d,s_d}}{r_{c,s_c}}.
\end{equation}
\end{lemma}

\proof The proof is analogous to that of Lemma~\ref{lem1}. We notice
that if $\deg x = s_x>s_d-s_c-1$ the degrees of each term in $Dx',
A, Bx$ and $Cx^2$ compare as follows:
$$
s_c+2s_x > s_d+s_x-1>s_b+s_x; \quad s_c+2s_x > s_a
$$
which is impossible as the maximum must be attained at at least two
of those values. For $s<s_d-s_c-1$, we get
$$
s_d+s-1 > s_c + 2s,\quad s_d+s-1>s_a,\quad s_d+s-1 > s_b+s.
$$
That is again impossible. Hence the only possibility is $s_x =
s_d-s_c-1$.

Now, comparing the leading terms gives us $s_xr_{d,s_d}r_{x,s_x} =
r_{c,s_c}r_{x,s_x}^2$ and~\eqref{eq_lem2} follows. \endproof

We end this section by showing that all the cubic polynomials
\textnumero~1--6 indeed have only one solution $x$ of positive
degree. This is the straightforward  implication of the following

\begin{lemma}\label{lem17}
Suppose that $b_0, b_1,b_2,b_3\in\QQ[t]$ satisfy $\deg(b_3)=0$,
$\deg(b_2)=1$ and $\deg(b_0)$, $\deg(b_1)\le 1$. then the equation
$b_3x^3+b_2x^2+b_1x+b_0=0$ has exactly one solution of strictly
positive degree.
\end{lemma}

\proof Notice that if $\deg x> 0$ then $\deg(b_0)$ and $\deg(b_1x)$
are strictly smaller than $\deg(b_2x^2)$. Suppose $\deg(x) = d>0$.
Then the degrees of the monomials $b_3x^3$ and $b_2x^2$ are $3d$ and
$1+2d$ respectively. Since the degrees of at least two monomials
among  $b_3x^3, b_2x^2, b_1x, b_0$ must coincide, we derive $d=1$.

Let $x = \sum_{n=-1}^\infty r_{-n} t^{-n}$ and $b_2 = \beta_1t +
\beta_2$. Then expanding the equation $b_3x^3+b_2x^2+b_1x+b_0=0$ and
comparing the terms at $t^3$ gives $r_1^2(b_3r_1 + \beta_1)=0$ or
$r_1 = -\beta_1/b_3$. Hence the term $r_1$ is uniquely determined.

Next, we consequently compare the coefficients at $t^{3-k}$ where
$k=1,2\ldots$ Each time we get a polynomial in $r_1, \ldots,
r_{1-k}$ where all terms except $r_{1-k}$ are already determined in
the previous steps. The terms involving $r_{1-k}$ are
$3b_3r_1^2r_{1-k} + 2\beta_1r_1r_{1-k} = -r_1\beta_1 r_{1-k}$. In
other words, we get a linear equation in $r_{1-k}$ which has the
unique solution. This implies that each term $r_1, r_0,\ldots$ of
$x$ is uniquely determined and therefore there is exactly one
solution $x$ with $\deg x>0$. \endproof

\subsection{The solution of $3x^3 - 3tx^2 - 9x + t = 0$}

As the first step, we use Proposition~\ref{prop1} to transfer this
equation to Riccati one. However, there is another, less technical
way to derive the Riccati equation for this particular cubic
equation: by differentiating $y = (1+\tau^{-1})^{1/3}$ one can
derive the Riccati equation for $y$ and then apply the transform $x
= \frac{1+y}{1-y}$ and $t = \frac{-\tau - 3}{6}$ which maps the
series $y(\tau)$ to $x(t)$. We leave the details to an interested
reader and proceed with the
formulae~\eqref{eq_ricd}--~\eqref{eq_ricc}:

$$D = -4\cdot 3^7 - 4\cdot 3^3t^4 +3^5t^2-3^6t^2-2\cdot 3^6t^2 =
-108(t^2+9)^2;$$
$$A = \frac13 \left(\begin{array}{ccc}
-3&-3t&-3t^2\\
0&18&18t\\
9&3t&9t^2+54
\end{array}\right) = -108(t^2+9);
$$
$$B = \frac13 \left(\begin{array}{ccc}
-9&-3&-3t^2\\
-6t&0&18t\\
9&9&9t^2+54
\end{array}\right) = 0;
$$
$$C = \left(\begin{array}{ccc}
-9&-3t&-3\\
-6&18&0\\
9&3t&9
\end{array}\right) = -108(t^2+9).
$$
We now divide both sides of the equation $Dx' = A+Bx+Cx^2$ by
$-108(t^2+9)$ and end up with the Riccati equation:
\begin{equation}\label{eq13}
(t^2+9)x' = 1+x^2.
\end{equation}
We already know that the initial cubic equation has exactly one
solution $x$ with $\deg x\ge 1$ and the other two solutions $y,z$
satisfy $\deg y, \deg z\le 0$. Therefore by Lemma~\ref{lem4}, $x$ is
the only solution of the Riccati equation with positive degree.
Hence we can freely apply Lemmata~\ref{lem1} and~\ref{lem2} for this
and subsequent equations.

We have $s_d=2$, $s_a=0$, $s_b = -\infty$ and $s_c=0$, therefore
Lemma~\ref{lem2} is applicable. Together with comparing the
coefficients of~\eqref{eq13} at $t$, it gives $[x] = t$ and $x_1(t)
= \frac{1}{x-t} = \phi_{t,1}(x)$.
%
Proposition~\ref{prop2} states that $x_1(t)$ satisfies the equation
$D_1 x_1' = A_1 + B_1x_1 + C_1x_1^2$ where
$$
\left(\begin{array}{c} A_1\\B_1\\C_1\\D_1
\end{array}\right) = \left(\begin{array}{cccc}
0&0&-1&0\\
0&-1&-2t&0\\
-1&-t&-t^2&1\\
0&0&0&1
\end{array}\right)\left(\begin{array}{c}
1\\0\\1\\t^2+9
\end{array}\right) = \left(\begin{array}{c}
-1\\-2t\\8\\t^2+9
\end{array}\right).
$$
%

Notice that $A_1$ and $B_1$ are constants, $C_1$ is an odd linear
function and $D_1$ is an even quadratic function. We will show that
the solution $x$ of any Riccati equation with these properties
enjoys an easily described continued fraction expansion. This result
may be of an independent interest.

\begin{proposition}\label{prop3}
Let $u_1,u_2,u_3,v_1,v_2\in \QQ$ be such that
\begin{enumerate}
\item $u_2, u_3, v_1\neq 0$;
\item $iv_1\neq u_2$ for all $i\in \NN$;
\item $(i^2 v_1 - iu_2)v_2 \neq u_1u_3$ for all $i\in \NN$.
\end{enumerate}
Let  $x\in \QQ((t^{-1}))$ be a solution of the equation $(v_1t^2 +
v_2)x' = u_1 + u_2tx + u_3x^2$ such that $\deg x\ge 1$. Then $x$ is
given by the formula
\begin{equation}\label{eq_prop3}
x(t) = \bigk \left[
\begin{array}{ccc}
u_3&\beta_1&\beta_2\\
\alpha_0t&\alpha_1 t&\alpha_2t
\end{array} \cdots
\right]
\end{equation}
where $\alpha_i = (2i+1)v_1 - u_2$ and $\beta_i = (i^2v_1 - iu_2)v_2
- u_1u_3$. 
\end{proposition}

\proof Let $x(t)$ have the continued fraction expansion $\bigk
\left[
\begin{array}{ccc}
u_3&\beta_1&\beta_2\\
a_0&a_1&a_2
\end{array} \cdots
\right] $ and $x_i(t)$ be the corresponding quotients $\bigk \left[
\begin{array}{ccc}
\beta_i&\beta_{i+1}&\beta_{i+2}\\
a_i&a_{i+1}&a_{i+2}
\end{array} \cdots
\right]$. Denote $u_3$ by $\beta_0$. Then by~\eqref{def_phi}, for
all $i\in \ZZ_{\ge 0}$, $\phi_{a_i, \beta_i} (x_i(t)) = x_{i+1}(t)$.
Let $x_i(t)$ satisfy the Riccati equation $D_i x_i' = A_i + B_ix_i +
C_i x_i^2$. We use Proposition~\ref{prop2} and Lemma~\ref{lem1} to
prove by induction that for all $i\in \NN$ one has $D_i = D = v_1
t^2 + v_2$, $A_i = -1$ ($i>0$), $B_i = (u_2 - 2i v_1)t $, $C_i =
(i^2v_1 - iu_2)v_2 - u_1u_3$ and $a_i = \alpha_i t$.

We start with $x(t)$. Since $\deg x>0$ and $u_2\neq v_1$, all the
conditions of Lemma~\ref{lem1} are satisfied and it implies that
$\deg x = 1$ and the leading coefficient $b_1$ of $x$ is $\frac{v_1
- u_2}{u_3}$. Comparing the coefficients for $t$ in the Riccati
equation gives that $b_0=0$ and hence $[x] = \frac{v_1 -
u_2}{u_3}t$. Since $v_1\neq u_2$ and $u_3\neq 0$, we can take $a_0 =
(v_1-u_2) t$ and $\beta_0 = u_3$. Then Proposition~\ref{prop2} gives
$$
\begin{array}{rl}
\left(\begin{array}{c} A_1\\B_1\\C_1\\D_1
\end{array}\right) &= \left(\begin{array}{cccc}
0&0&-1&0\\
0&-u_3&-2(v_1-u_2)t&0\\
-u_3^2&-u_3(v_1-u_2)t&-(v_1-u_2)^2t^2&(v_1-u_2)u_3\\
0&0&0&u_3
\end{array}\right)\left(\begin{array}{c}
u_1\\u_2t\\u_3\\v_1t+v_2
\end{array}\right)\\[6ex]
& = \left(\begin{array}{c} -u_3\\u_3(u_2 - 2v_1)t\\u_3((v_1-u_2)v_2
- u_1u_3)\\u_3(v_1t^2 + v_2)
\end{array}\right).
\end{array}
$$
By dividing all terms of the resulting Riccati equation by $u_3$ we
achieve the base of induction.

Now assume that the assumptions are satisfied for $i\in\NN$ and
verify them for $i+1$. Since the equation $(v_1t^2 + v_2)x'_i = -1 +
(u_2-2iv_1)t x_i + ((i^2 v_1 - iu_2)v_2 - u_1u_3)x_i^2$ for $x_i$ is
of the same form as for the initial series $x$ and all three
conditions for its coefficients are satisfied, we use the same
arguments as in the proof of the base of induction and get $a_i =
(v_1 - (u_2 - 2iv_1))t = ((2i+1)v_1 - u_2)t$ and $\beta_i = (i^2 v_1
- iu_2)v_2 - u_1u_3$. Then Proposition~\ref{prop2} gives that the
coefficients of the Riccati equation for $\phi_{a_i, \beta_i}(y_i)$
are: $A_{i+1}=-1$, $B_{i+1} = (u_2 - 2iv_1 - 2v_1)t = (u_2 -
2(i+1)v_1)t$; $C_{i+1} = (v_1 - (u_2-2iv_1))v_2 + (i^2 v_1 -
iu_2)v_2 - u_1u_3 = ((i+1)^2 v_1 - (i+1)u_2)v_2 - u_1u_3$. The
inductional step is verified.
\endproof

We now substitute the values $u_1=-1, u_2=-2, u_3=8, v_1=1, v_2=9$
into Proposition~\ref{prop3} and get $a_{i+1} = (2i+1+2)t =
(2(i+1)+1)t$, $\beta_{i+1} = 9(i^2+2i)+8 = 9(i+1)^2-1$. The
continued fraction \textnumero~1 is established.

\subsection{The solution of $3x^3-3tx^2-3ax+at=0$}\label{subsec32}

As in the previous case, we first transfer the algebraic equation
for $x$ into the Riccati one. Equations~~\eqref{eq_ricd}
--~\eqref{eq_ricc} give us
$$D = -4\cdot 3^4a^3-4\cdot 3^3at^4 + 3^5a^2t^2-3^4a^2t^2-2\cdot 3^5 a^2t^2 =
-108(3a^3+3a^2t^2+at^4)
$$
$$A = \frac13 \left(\begin{array}{ccc}
-3a&-3at&-3at^2\\
0&6a&0\\
9&3t&9t^2+18a
\end{array}\right) = -108a^3;
$$
$$B = \frac13 \left(\begin{array}{ccc}
-3a&-3a&-3at^2\\
-6t&0&0\\
9&9&9t^2+18a
\end{array}\right) = -108 a^2t;
$$
$$C = \left(\begin{array}{ccc}
-3a&-3at&-3a\\
-6t&6a&0\\
9&3t&9
\end{array}\right) = -108at^2.
$$
Then divide all parts of the equation by $-108a$ and get
$(3a^2+3at^2+t^4)x' = a^2+atx+t^2x^2$. Application of
Lemma~\ref{lem2} together with the comparison of the coefficients at
$t^3$ give $[x] = t$ and $x_1(t) = \phi_{t,1}(x)$. We apply
Proposition~\ref{prop2} to compute the coefficients of the equation
$D_1x_1' = A_1+B_1x_1+C_1x_1^2$:
$$
\left(\begin{array}{c} A_1\\B_1\\C_1\\D_1
\end{array}\right) = \left(\begin{array}{cccc}
0&0&-1&0\\
0&-1&-2t&0\\
-1&-t&-t^2&1\\
0&0&0&1
\end{array}\right)\left(\begin{array}{c}
a^2\\at\\t^2\\t^4+3at^2+3a^2
\end{array}\right) = \left(\begin{array}{c}
-t^2\\-2t^3-at\\2at^2+2a^2\\t^4+3at^2+3a^2
\end{array}\right).
$$

{\bf Remark.} Notice that $A_1, C_1$ and $D_1$ are even polynomials
in $t$, while $B_1$ is odd. Application of Lemma~\ref{lem1} together
with the comparison of the coefficients at $t^3$ gives that $[x_1]$
is a constant multiple of $t$. Then a quick observation of the
formula in Proposition~\ref{prop2} tells us that the coefficients
$A_2,C_2,D_2$ are still even, while $B_2$ is odd. From these
observation one can deduce the following interesting fact: if the
coefficients $A, C$ of the Riccati equation are even quadratic
polynomials, $D$ is even polynomial of degree 4 and $B$ is odd
polynomial of degree 3 and all partial quotients of the continued
fraction of the solution $x$ are linear then they are all constant
multiples of $t$.

Next, we will show by induction the following formulae for the
coefficients $A_i, B_i, C_i$ and the partial quotients $a_i,
\beta_i$ of $x_i$: {\footnotesize
\begin{equation}\label{no5_eq1}
A_{4k+1} = -(12k+1)t^2 - 18ka,\; B_{4k+1} = -(12k+2)t^3 -
(12k+1)at,\; C_{4k+1} = (6k+2)at^2+(6k+2)a^2,
\end{equation}
$$
 a_{4k+1} = 3(4k+1)t,
\beta_{4k+1} = (6k+2)a;
$$
\begin{equation}\label{no5_eq2}
A_{4k+2} = -t^2-a,\; B_{4k+2} = -(12k+4)t^3 - (12k+5)at,\; C_{4k+2}
= (6k+1)(12k+5)at^2 + 9(6k+1)a^2;
\end{equation}
$$
a_{4k+2} = t,\; b_{4k+2} = (6k+1)a;
$$
\begin{equation}\label{no5_eq3}
A_{4k+3} = -(12k+5)t^2-9(2k+1)a,\; B_{4k+3} = -(12k+6)t^3 -
(24k+13)at,\; C_{4k+3} = (6k+4)a^2;
\end{equation}
$$
a_{4k+3} = 3(4k+3)(t^3+2at),\; \beta_{4k+3} = (6k+4)a^2.
$$
\begin{equation}\label{no5_eq4}
A_{4k+4} = -1,\; B_{4k+4} = -(12k+12)t^3 -
(24k+23)at,\; C_{4k+4} = (6k+5)a^2((12k+13)t^2 + (18k+18)a);
\end{equation}
$$
a_{4k+4} = t,\; \beta_{4k+4} = (6k+5)a^2.
$$
 }

Indeed, already computed values of $A_1,B_1,C_1$ constitute the base
of induction. Suppose now that for all $1\le i\le 4k+1$ the values
of $A_i,B_i$ and $C_i$ are given by the formulae above. Then
Lemma~\ref{lem1} is applicable for $x_{4k+1}$ and~\eqref{eq_lem1}
gives $[x_{4k+1}] = \frac{12k+3}{(6k+2)a}t$, thus we choose $a_i =
3(4k+1)t, \beta_i = (6k+2)a$. Next, Proposition~\ref{prop2} infers
{\footnotesize
$$
\left(\begin{array}{cccc}
0&0&-1&0\\
0&-(6k+2)a&-6(4k+1)t&0\\
\hspace{-1ex}-(6k+2)^2a^2&\hspace{-1ex}-3(6k+2)(4k+1)at&\hspace{-1ex}-9(4k+1)^2t^2&\hspace{-1.6ex}3(4k+1)(6k+2)a\hspace{-1ex}\\
0&0&0&(6k+2)a
\end{array}\right)\left(\begin{array}{c}
-(12k+1)t^2-18ka\\\hspace{-1ex}-(12k+2)t^3-(12k+1)at\hspace{-1ex}\\\hspace{-1ex}(6k+2)at^2+(6k+2)a^2\hspace{-1ex}\\t^4+3at^2+3a^2
\end{array}\right)
$$$$
= \left(\begin{array}{c} -(6k+2)a(t^2+a)\\-(6k+2)a((12k+4)t^3 +
(12k+5)at)\\(6k+2)a((6k+1)(12k+5)at^2 +9(6k+1)(2k+1)a^2)
\\(6k+2)a(t^4+3at^2+3a^2)
\end{array}\right).
$$
} After division of both sides of the new equation by $(6k+2)a$ we
get the values $A_{4k+2}, B_{4k+2}$ and $C_{4k+2}$ as
in~\eqref{no5_eq2}. Then application of Lemma~\ref{lem1} gives
$[x_{4k+2}]= \frac{12k+5}{(6k+1)(12k+5)a}t$ and we choose $a_{4k+2}
= t, \beta_{4k+2} = (6k+1)a$.

We apply Proposition~\ref{prop2} once again to get {\footnotesize
$$
\left(\begin{array}{cccc}
0&0&-1&0\\
0&-(6k+1)a&-2t&0\\
\hspace{-1ex}-(6k+1)^2a^2&\hspace{-1ex}-(6k+1)at&\hspace{-1ex}-t^2&\hspace{-1.6ex}(6k+1)a\hspace{-1ex}\\
0&0&0&(6k+1)a
\end{array}\right)\left(\begin{array}{c} -t^2-a\\-(12k+4)t^3 -
(12k+5)at\\(6k+1)(12k+5)at^2 +9(6k+1)(2k+1)a^2
\\t^4+3at^2+3a^2
\end{array}\right).
$$$$
= \left(\begin{array}{c}
-(6k+1)a((12k+5)t^2+9(2k+1)a)\\-(6k+1)a((12k+6)t^3 +
(24k+13)at)\\(6k+1)(6k+4)a^3
\\(6k+1)a(t^4+3at^2+3a^2)
\end{array}\right).
$$}
\!\!Dividing both sides of the new equation by $(6k+1)a$ gives the
values $A_{4k+3}, B_{4k+3}$ and $C_{4k+3}$ from~\eqref{no5_eq3}.

We apply Lemma~\ref{lem1} to get that $[x_{4k+3}]$ is a cubic
polynomial with the leading coefficient $\frac{3+12k+6
}{(6k+4)a^2}$. After the comparison of the coefficients at both
sides of the Riccati equation for $t^5, t^4, t^3$, we get that
$$
[x_{4k+3}] = \frac{3(4k+3)}{(6k+4)a^2} (t^3 + 2at)
$$
and therefore we choose $a_{4k+3} = 3(4k+3)(t^3+2at)$, $\beta_{4k+3}
= (6k+4)a^2$.

We apply Proposition~\ref{prop2} the third time {\footnotesize
$$
\left(\begin{array}{cccc}
0&0&-1&0\\
0&-(6k+4)a^2&-6(4k+3)(t^3+2at)&0\\
\hspace{-1ex}-(6k+4)^2a^4&\hspace{-1ex}-(6k+4)(12k+9)a^2(t^3+2at)&\hspace{-1ex}-(12k+9)^2(t^3+2at)^2&\hspace{-1.6ex}3(4k+3)(6k+4)a^2(3t^2+2a)\hspace{-1ex}\\
0&0&0&(6k+5)a^2
\end{array}\right)
$$$$
\times \left(\begin{array}{c} -(12k+5)t^2-9(2k+1)a\\-(12k+6)t^3 -
(24k+13)at\\(6k+4)a^2
\\t^4+3at^2+3a^2
\end{array}\right)
= \left(\begin{array}{c} -(6k+4)a^2\\-(6k+4)a^2((12k+12)t^3 +
(24k+23)at)\\(6k+4)a^2((6k+5)(12k+13)a^2t^2 + 18(6k+5)(k+1)a^3)
\\(6k+4)a^2(t^4+3at^2+3a^2)
\end{array}\right).
$$
} Dividing both sides of the Riccati equation by $(6k+4)a^2$ gives
the values $A_{4k+4}, B_{4k+4}$ and $C_{4k+4}$ from~\eqref{no5_eq4}.

Lemma~\ref{lem1} infers that $[x_{4k+4}] = \frac{12k+13}{(6k+5)(12k+13)a^2}t$
and we choose $a_{4k+4}=t$,\\ $\beta_{4k+4}$ $=(6k+5)a^2$.

We apply Proposition~\ref{prop2} the last time to get {\footnotesize
$$
\left(\begin{array}{cccc}
0&0&-1&0\\
0&-(6k+5)a^2&-2t&0\\
\hspace{-1ex}-(6k+5)^2a^4&\hspace{-1ex}-(6k+5)a^2t&\hspace{-1ex}-t^2&\hspace{-1.6ex}(6k+5)a^2\hspace{-1ex}\\
0&0&0&(6k+4)a^2
\end{array}\right) \left(\begin{array}{c} -1\\-(12k+12)t^3 -
(24k+23)at\\(6k+5)(12k+13)a^2t^2 + 18(6k+5)(k+1)a^3
\\t^4+3at^2+3a^2
\end{array}\right)
$$$$
= \left(\begin{array}{c}
-(6k+5)a^2((12k+13)t^2+18(k+1)a)\\-(6k+5)a^2((12k+14)t^3 +
(12k+13)at)\\(6k+5)a^2((6k+8)at^2 + (6k+8)a^2)
\\(6k+5)a^2(t^4+3at^2+3a^2)
\end{array}\right).
$$
}
After division of both sides of the resulting equation by $(6k+5)a^2$, we
derive the values $A_{4k+5}, B_{4k+5}$ and $C_{4k+5}$ from~\eqref{no5_eq1}.
That finishes the inductional step.

\section{Application to rational approximations of cubic
irrationals}\label{sec5}

In this section we prove Theorem~\ref{th3}. But first of all, we
need to separate the root of the cubic equation~\textnumero 5 whose
continued fraction is~\eqref{no5} from the other two.

\begin{lemma}\label{lem15}
Let $a,t\in\CC$ be such that $|t|^2>9|a|>0$. Then the equation
\begin{equation}\label{lem15_eq}
3x^3 - 3tx^2 -3ax +at=0
\end{equation}
has exactly one root $\xi$ such that $|\xi| > |a|^{1/2}$. Moreover,
if the continued fraction~\textnumero5 converges then it converges
to $\xi$.
\end{lemma}

\proof Let $a$ be fixed. The leading term of the series $x(t)$
from~\eqref{no5} is $t$. Therefore, $\lim_{t\to\infty}
|x(t)|=\infty$. Moreover, for $t>T_0$, where $T_0$ is the radius of
convergence of $x(t)$, this function is analytic and therefore
continuous. Hence, the continued fraction~\eqref{no5} of $x(t)$
corresponds to a root of the equation $3x^3 - 3tx^2-3ax + at=0$
which continuously tends to infinity as $t\to\infty$.

Let $x$ be any complex number that satisfies $|x| = |a|^{1/2}$. We have
$$|3x^3 - 3ax| < 3|a|^{1/2}\cdot 2|a| =
6|a|^{3/2}.$$ On the other hand, $|3tx^2-at| \ge |t|\cdot2|a| >
6|a|^{3/2}$. Therefore the number of roots of~\eqref{lem15_eq}
inside $|x|< |a|^{1/2}$ is the same as one for $3tx^2-at=0$, i.e. it
is two. We thus get that there is only one root $\xi$
of~\eqref{lem15_eq} outside that circle. Then, as $t\to\infty$, this
root $\xi(t)$ continuously tends to infinity, because the other two
can not leave the circle $|x|<|a|^{1/2}$.\endproof


As the next step, we show that the continued fraction~\eqref{no5}
uniformly converges in some neighbourhood $U_\infty\in\CC$ of
infinity. As discussed in Section~\ref{ssec_1_2}, that will imply
that the continued fraction~\eqref{no5}, computed at $t\in
U_\infty$, will converge to the solution $x(t)\in\CC$ of a cubic
equation~\textnumero 5.

\begin{lemma}
Let $a\in \CC$ be fixed. Then the continued fraction~\eqref{no5}
uniformly converges for all $t\in U_\infty = \{t\in\CC: |t|\ge
12|a|^2\}$.
\end{lemma}

\begin{proof}
For each $k\in \ZZ$ we simultaneously divide the values
$\beta_{4k+1}, \beta_{4k+2}$ and $a_{4k+1}$ by $k$. and also divide
$\beta_{4k+3}, \beta_{4k+4}$ and $a_{4k+3}$ by $k$. As was verified
in Lemma~\ref{lem8}, such transformations do not change the
convergence and the limit of the continued fraction. For a modified
continued fraction, one can readily verify that $\forall t\in
U_\infty$,
$$
3(4 + k^{-1})|t| > 2(3+k^{-1})|a| + 2;\quad |t| > (6 + k^{-1})|a| +
2;
$$$$
3(4 + 3k^{-1})|t|\cdot |t^2 + 2a| > 2(3+2k^{-1})|a|^2+2;\quad
|t|>(6+ 5k^{-1})|a|^2+2.
$$
Then the easy modification of Pringsheim convergence
criteria~\cite[Theorem 4.35]{jones_thron} implies that
$$
\bigk\left[ \begin{array}{l}  \\ 0 \end{array} \overline{\begin{array}{llll}2(3+k^{-1})a&(6+k^{-1})a&2(3+2k^{-1})a^2& (6+5k^{-1})a^2\\
3(4+k^{-1})t&t&3(4+3k^{-1})t(t^2+2a)&t
\end{array}\hspace{-1ex}}\;\; \right]
$$
uniformly converges in $U_\infty$.
\end{proof}

We will now have a closer look at the convergence of~\eqref{no5} in an open
neighbourhood of $\RR_{\ge 1}$. We will see that the continued fraction
uniformly converges in this region and therefore it can be analytically
extended from $U_\infty$ to contain $\RR_{\ge 1}$. For the rest of this
section we assume that $a$ is a positive integer and $t\in \RR_{\ge 1}$.

Let $\xi = x(t,a)$ be the cubic irrational number from
Lemma~\ref{lem15}. Let $p_n/q_n = p_n(t,a)/q_n(t,a)$ be the $n$'th
convergent of the continued fraction~\eqref{no5}. We want to
estimate the size of $|\xi - p_n/q_n|$. We will do that for values
of $n$ of the form $n = 4k+2$, $k\in\ZZ$.

Denote
$$
S_k:= \left(\begin{array}{cc} p_{4k+2}&q_{4k+2}\\
p_{4k+1}&q_{4k+1}
\end{array}\right);\quad T_k:= \left(\begin{array}{cc} p_{4k+2}&q_{4k+2}\\
p_{4k-2}&q_{4k-2}
\end{array}\right);\quad U_k:= \left(\begin{array}{cc} p_{4k-1}&q_{4k-1}\\
p_{4k-2}&q_{4k-2}
\end{array}
\right).
$$
From~\eqref{no5} and the recurrent formulae for convergents we get that
$S_{k+1} = A_kS_k$ where {\footnotesize
\begin{equation}\label{eq14} A_k =
\left(\begin{array}{cc} \!\!t&\hspace{-1.5ex} (6k+7)a\!\!\\
1&0
\end{array}
\right)
\left(\begin{array}{cc} \!\!3(4k+5)t&\hspace{-1.5ex}(6k+8)a\!\!\\
1&0
\end{array}
\right)
\left(\begin{array}{cc} \!\!t&\hspace{-1.5ex}(6k+5)a^2\!\!\\
1&0
\end{array}
\right)
\left(\begin{array}{cc} \!\!3(4k+3)t(t^2+2a)&\hspace{-1.5ex} (6k+4)a^2\!\!\\
1&0
\end{array}
\right).
\end{equation}}
For convenience, we denote the four matrices in the product above by
$C_{4k+1}, C_{4k+2}, C_{4k+3}$ and $C_{4k+4}$ so that $A_k =
C_{4k+4}C_{4k+3}C_{4k+2}C_{4k+1}$. Then, the relation between $T_{k+1}$ and
$S_k$ is
\begin{equation}\label{eq3}
T_{k+1} = \left(\begin{array}{cc}
a_{k11}&a_{k12}\\
1&0
\end{array}
\right) S_k
\end{equation}
where $a_{k11}$ and $a_{k12} $ are the corresponding entries of the
matrix $A_k$. One then computes (we used Wolfram Mathematika for
convenience) {\small
\begin{equation}\label{eq7}
\begin{array}{l}
a_{k11} = 9(4k+3)(4k+5)t^2(t^2+a)(t^2+2a) + 3(4k+5)(6k+5)a^2t^2 +
(6k+5)(6k+7) a^3,\\[1ex]
a_{k12} = 6(4k+5)(3k+2)ta^2(t^2+a).
\end{array}
\end{equation}}

Another application of the recurrent formula for convergents and~\eqref{no5}
gives $S_k = C_{4k}C_{4k-1}C_{4k-2}U_k$ or
$$
U_k = \frac{-1}{(6k-1)(6k+1)(6k+2)a^4} B_kS_k =: d^{-1}B_kS_k
$$
where
$$
B_k = \left(\begin{array}{cc} 0&\hspace{-1.5ex}-(6k-1)a^2\\
-1&t
\end{array}
\right) \left(\begin{array}{cc} 0&\hspace{-1.5ex}-(6k+2)a\\
-1&3(4k+1)t
\end{array}\right) \left(\begin{array}{cc} 0&\hspace{-1.5ex}-(6k+1)a\\
-1&t
\end{array}
\right).
$$
Now we can get the relation between $T_k$ and $S_k$:
\begin{equation}\label{eq4}
T_k = \left(\begin{array}{cc} 1&0\\
b_{k21}/d&b_{k22}/d
\end{array}
\right) S_k.
\end{equation}
where $b_{k21}$ and $b_{k22} $ are the corresponding entries of the
matrix $B_k$. The computations show that
\begin{equation}\label{eq6}
b_{k21} = -(12k+3)t^2 - (6k+2)a,\quad b_{k22} = 3(4k+1)t(t^2+a).
\end{equation}

Finally,~\eqref{eq3} and~\eqref{eq4} provide the relation between $T_k$ and
$T_{k+1}$:
\begin{equation}\label{eq40}
T_{k+1} = \frac{d}{b_{k22}} \left(\!\!\begin{array}{cc}
a_{k11}&a_{k12}\\
1&0
\end{array}
\!\!\right)\!\!  \left(\!\!\begin{array}{cc}
b_{k22}/d&0\\
-b_{k21}/d&1
\end{array}
\right) T_k = \left(\!\!\begin{array}{cc}
a_{k11} - a_{k12} \frac{b_{k21}}{b_{k22}}&\frac{da_{k12}}{b_{k22}}\\
1&0
\end{array}
\!\!\right)\! T_k=: D_kT_k.
\end{equation}
Now we compute
\begin{equation}\label{eq9}
\frac{da_{k12}}{b_{k22}} =
\frac{-(4k+5)(6k-1)(6k+1)(6k+2)(6k+4)a^6}{4k+1}.
\end{equation}

\begin{lemma}\label{lem6}
Let $a\in \ZZ$ be positive and $t\in \RR_{\ge 1}$. Then the
denominators $q_{4k+2}$ and $q_{4k+6}$ satisfy the inequalities
\begin{equation}\label{lem6_eq}
9(4k+3)(4k+5)t^2(t^2+a)(t^2+2a)
q_{4k+2}<q_{4k+6}<9(4k+3)(4k+5)(t^2+a)^3q_{4k+2}
\end{equation}
\end{lemma}

\proof First of all, since all the entries of the matrices $C_i$ are
positive, we immediately get that $q_n>0$ for all $n\in\NN$. The
first inequality follows immediately from~\eqref{eq3}
and~\eqref{eq7}. For the second expression we observe that
$t^2(t^2+a)(t^2+2a) = (t^2+a)^3 - t^2a^2 - a^3$ and
therefore~\eqref{eq7} can be rewritten as
\begin{equation}\label{eq8}
a_{k11}=9(4k+3)(4k+5)(t^2+a)^3 -
6(4k+5)(3k+2)t^2a^2-4(27k^2+54k+25)a^3.
\end{equation}
We also notice that $q_{4k+2} = tq_{4k+1} + (6k+1)aq_{4k}$,
therefore~\eqref{eq7} implies
$$a_{k12}q_{4k+1} = 6(4k+5)(3k+2)a^2(t^2+a)(q_{4k+2} - (6k+1)aq_{4k})< 6(4k+5)(3k+2)a^2(t^2+a)q_{4k+2}.$$
Finally, we substitute the last inequality and~\eqref{eq8}
into~\eqref{eq3} and get
$$
q_{4k+6}< (9(4k+3)(4k+5)(t^2+a)^3 - (6k+5)(6k+8)a^3)q_{4k+2},
$$
and the second inequality in~\eqref{lem6_eq} follows
immediately.\endproof

Now we are ready to provide the lower and upper bounds for $q_{4k+2}$. In
view of~\eqref{no5}, one can easily compute that $q_2=3t^2+a$. Also, the
Stirling's formula together with the evaluation of the Wallis integral gives
that $\forall k\in \ZZ_{\ge 2}$,
$$
2k\left(\frac{2k}{e}\right)^k<(2k+1)!!<
4k\left(\frac{2k}{e}\right)^k.
$$

We can now get the estimates on the denominators $q_{4k+2}$:
\begin{equation}\label{eq11}
q_{4k+2}<(3t^2+a)(4k+1)!! (9(t^2+a)^3)^k < 8(3t^2+a)k c_6^k\cdot
k^{2k}\quad \mbox{and}
\end{equation}
\begin{equation}\label{eq10}
q_{4k+2}>(3t^2+a)(4k+1)!! (9t^2(t^2+a)(t^2+2a))^k > 4(3t^2+a)k
c_7^k\cdot k^{2k},
\end{equation}
where $c_6 = 144(t^2+a)^3/e^2$ and $c_7 =
144t^2(t^2+a)(t^2+2a)/e^2$.

In the next step, we estimate
$\frac{p_{4k+6}}{q_{4k+6}}-\frac{p_{4k+2}}{q_{4k+2}}$.

\begin{lemma}\label{lem7}
For all $a\in\NN$ and $t\ge 1$ with $a^3<t^2(t^2+a)(t^2+2a)$ the
continued fraction~\eqref{no5} converges. Moreover, its limit $\xi =
x(t,a)$ satisfies
\begin{equation}\label{eq_lem7}
\left|\frac{p_{4k+2}}{q_{4k+2}} - \frac{p_{4k+6}}{q_{4k+6}}\right|<
\left|\xi - \frac{p_{4k+2}}{q_{4k+2}}\right| <
2\left|\frac{p_{4k+2}}{q_{4k+2}} - \frac{p_{4k+6}}{q_{4k+6}}\right|<
\tau_3 c_4^k,
\end{equation}
where
$$\tau_3 = \frac{105\sqrt{3}e^2a^4}{8 t(t^2+2a)(3t^2+a)^2}\quad
\mbox{and}\quad c_4 =\frac{a^6}{16t^4(t^2+a)^2(t^2+2a)^2}.$$ In particular,
the continued fraction converges uniformly in $t$.
\end{lemma}
\proof Equation~\eqref{eq9} together with~\eqref{eq40} give
$$
\frac{p_{4k+6}}{q_{4k+6}} - \frac{p_{4k+2}}{q_{4k+2}} =
\frac{(4k+5)(6k-1)(6k+1)(6k+2)(6k+4)a^6q_{4k-2}q_{4k+2}}{(4k+1)q_{4k+2}q_{4k+6}}\left(\frac{p_{4k+2}}{q_{4k+2}}
- \frac{p_{4k-2}}{q_{4k-2}}\right)
$$
In particular, that means that differences
$\frac{p_{4k+6}}{q_{4k+6}} - \frac{p_{4k+2}}{q_{4k+2}}$ share the
same sign for all $k\in\NN$. By Lemma~\ref{lem6}, the big fraction
in the last equation is bounded from above by
\begin{equation}\label{eq41}
\frac{(4k+5)(6k-1)(6k+1)(6k+2)(6k+4)a^6}{(4k+1)\cdot 9^2
(4k-1)(4k+1)(4k+3)(4k+5)(t^2(t^2+a)(t^2+2a))^2}<\frac12.
\end{equation}
One computes $\frac{p_{4k+2}}{q_{4k+2}}$ as the sum
$$
\frac{p_{4k+2}}{q_{4k+2}} = \frac{p_2}{q_2} + \sum_{i=1}^k
\left(\frac{p_{4i+2}}{q_{4i+2}} - \frac{p_{4i-2}}{q_{4i-2}}\right).
$$
Since the terms in the summation decrease by a factor bigger than
two, the right hand side has a limit as $k\to\infty$ and moreover,
$$
\xi = \lim_{k\to\infty} \frac{p_{4k+2}}{q_{4k+2}} < \frac{p_2}{q_2}
+ 2\left(\frac{p_{6}}{q_6} - \frac{p_2}{q_2}\right).
$$
Now, the first two inequalities in~\eqref{eq_lem7} readily follow
from
$$
\xi = \frac{p_{4k+2}}{q_{4k+2}} + \sum_{i=k}^\infty
\left(\frac{p_{4i+6}}{q_{4i+6}} - \frac{p_{4i+2}}{q_{4i+2}}\right).
$$
For the remaining inequality in~\eqref{eq_lem7}, we use~\eqref{eq40}
to get
$$
p_{4k+6}q_{4k+2}-p_{4k+2}q_{4k+6} = \prod_{i=0}^k\det D_{i+1} \det
T_1
$$$$
= \frac{4k+5}{5}\cdot a^{6k}\prod_{i=1}^{k} (6i-1)(6i+1)(6i+2)(6i+4)\det T_1
$$$$
=\frac{(4k+5)a^{6k}}{5\cdot 2\cdot 4} \cdot (6k+1)(6k+2)(6k+4)\cdot
\frac{(6k)!}{3^{2k}(2k)!}\det T_1.
$$
The last expression is bounded from above by
$$
< 126k^4a^{6k}\cdot \frac{\sqrt{12\pi
k}(6k/e)^{6k}}{3^{2k}\sqrt{4\pi k} (2k/e)^{2k}}\det T_1 =
126\sqrt{3}k^4 \cdot c_5^k\cdot k^{4k}\det T_1.
$$
where $c_5 = 6^4a^6/e^4$. To compute $\det T_1$, we use~\eqref{eq3} which
gives
$$\det T_1 = -60ta^2(t^2+a) \det S_0.$$
Direct computations then reveal
$$
p_1=3t^2+2a,\; p_2 = 3t(t^2+a),\; q_1=3t,\; q_2=3t^2+a
$$
and therefore $\det S_0 = -2a^2$. Combining all that information
together, we get
$$
\det T_{k+1} < 15120\sqrt{3}ta^4(t^2+a)k^4\cdot c_5^k\cdot k^{4k}.
$$

The second inequality in~\eqref{eq_lem7} implies
$$
\left| \xi-\frac{p_{4k+2}}{q_{4k+2}}\right| < 2\frac{\det
T_{k+1}}{q_{4k+2}q_{4k+6}} \stackrel{\eqref{eq10}}< \frac{2\cdot
15120\sqrt{3}ta^4(t^2+a)}{16 (3t^2+a)^2c_7}\cdot
\left(\frac{c_5}{c_7^2}\right)^k
$$$$
= \frac{105\sqrt{3}e^2a^4}{8 t(t^2+2a)(3t^2+a)^2}
\left(\frac{c_5}{c_7^2}\right)^k.
$$
Finally, a direct computation implies that $c_5/c_7^2 = c_4$ and
hence the last inequality in~\eqref{eq_lem7} is verified.
\endproof

Notice that if we do not restrict $t$ to positive real numbers, then the
sequence of the convergents $\frac{p_{4k+2}}{q_{4k+2}}$ considered as
functions of $t$ converges uniformly as soon as the absolute value of the
fraction~\eqref{eq41} together with the value of $|c_4|$ is strictly less
than some absolute constant $\sigma<1$. For all $t\in\RR$ that satisfy the
conditions of Lemma~\ref{lem7} there exists an open neighbourhood
$U(t)\in\CC$ of $t$ where those upper bounds are satisfied. Therefore, we can
analytically continue the limit of the continued fraction~\eqref{no5} from
$U_\infty$ to
$$U_\infty \cup \bigcup_{t} U(t).$$
This observation justifies that the continued fraction~\eqref{no5} converges
to the root of the cubic equation~\textnumero 5 for all positive integer $a$
and all $t\in\RR_{\ge 1}$ such that $t^2(t^2+a)(t^2+2a)>a^3$. Notice that the
last inequality follows from the condition $t^2\ge9a$ of Theorem~\ref{th3}.

In the further discussion we always assume that $t\in\NN$. We then
provide a series of good approximations to $\xi$ which will later be
used to show that there are no too good rational approximations of
$\xi$.

\begin{lemma}\label{lem5}
For all $t\in\ZZ$ and $k\in\ZZ_{\ge 2}$,
\begin{equation}\label{lem5_eq}
\gcd(p_{4k+2},q_{4k+2}) \ge \prod_{p\in\PP\atop p\ge 5}
p^{\left\lfloor \frac{2k}{p}\right\rfloor}.
\end{equation}
\end{lemma}

\proof Consider a positive integer $m$ and any integer divisor $p\mid 4m+5$.
To save the space we denote $r:=t(t^2+2a)$. Then $m\equiv -5/4$ (mod $p$) and
one can check that for all integer $l$
$$
C_{4(m+l)+4} = \left(\begin{array}{cc} t&\hspace{-1ex}(6(m+l)+7)a\\
1&0
\end{array}
\right) \equiv \left(\begin{array}{cc} t&(6l-1/2) a\\
1&0
\end{array}
\right)\; (\mathrm{mod}\;p);
$$
$$
C_{4(m+l)+3} = \left(\begin{array}{cc} 3(4(m+l)+5)t&\hspace{-1ex}(6(m+l)+8)a\\
1&0
\end{array}
\right) \equiv \left(\begin{array}{cc} 12lt&(6l+1/2) a\\
1&0
\end{array}
\right)\; (\mathrm{mod}\;p);
$$
$$
C_{4(m+l)+2} = \left(\begin{array}{cc} t&\hspace{-1ex}(6(m+l)+5)a^2\\
1&0
\end{array}
\right) \equiv \left(\begin{array}{cc} t&(6l-5/2) a^2\\
1&0
\end{array}
\right)\; (\mathrm{mod}\;p);
$$
$$
C_{4(m+l)+1} = \left(\begin{array}{cc} 3(4(m+l)+3)r&\hspace{-1ex}(6(m+l)+4)a^2\\
1&0
\end{array}
\right) \equiv \left(\begin{array}{cc} (12l-6)r&(6l-7/2) a^2\\
1&0
\end{array}
\right)\; (\mathrm{mod}\;p).
$$

We will show by induction by $l$ that
$$
\mathbf{C}_l:=\prod_{k=4m+3+l}^{4m+3-l} C_k \equiv \left(\begin{array}{cc} 0&\prod_{i=0}^l c_i\\
\prod_{i=0}^{l-1}c_i&0
\end{array}
\right)\; (\mathrm{mod}\;p),\quad \mbox{where}\; \left\{
\begin{array}{l}
c_{4i} = -(6i-1/2)a\\
c_{4i+1} = (6i+5/2)a^2\\
c_{4i+2} = -(6i+7/2)a^2\\
c_{4i+3} = (6i+13/2)a
\end{array}\right. .
$$
Indeed, the base of induction $l=0$ is straightforward and the inductional
step is verified by the series of matrix multiplications: {\footnotesize
$$
C_{4(m+j)+4}\mathbf{C_{4j}}C_{4(m-j)+2} \equiv \left(\hspace{-1ex}\begin{array}{cc} t&\hspace{-1ex}(6j-1/2)a\\
1&0 \end{array}\hspace{-1ex}\right)\left(\hspace{-1ex}\begin{array}{cc} 0&\hspace{-1.5ex}\prod_{i=0}^{4j} c_i\\
\prod_{i=0}^{4j-1}c_i&0 \end{array}\hspace{-1ex}\right)\left(\hspace{-1ex}\begin{array}{cc} t&\hspace{-1ex}-(6j+5/2)a^2\\
1&0 \end{array}\hspace{-1ex}\right)\equiv \left(\hspace{-1ex}\begin{array}{cc} 0&\hspace{-1.5ex}\prod_{i=0}^{4j+1} c_i\\
\prod_{i=0}^{4j}c_i&0 \end{array}\hspace{-1ex}\right);
$$
$$
\begin{array}{rl}
C_{4(m+j)+5}\mathbf{C_{4j+1}}C_{4(m-j)+1} &\equiv \left(\hspace{-1ex}\begin{array}{cc} (12j+6)r&\hspace{-1ex}(6j+5/2)a^2\\
1&0 \end{array}\hspace{-1ex}\right)\left(\hspace{-1ex}\begin{array}{cc} 0&\hspace{-1.5ex}\prod_{i=0}^{4j+1} c_i\\
\prod_{i=0}^{4j}c_i&0 \end{array}\hspace{-1ex}\right)\left(\hspace{-1ex}\begin{array}{cc} -(12j+6)r&\hspace{-1ex}-(6j+7/2)a^2\\
1&0 \end{array}\hspace{-1ex}\right)\\
&\equiv \left(\hspace{-1ex}\begin{array}{cc} 0&\hspace{-1.5ex}\prod_{i=0}^{4j+2} c_i\\
\prod_{i=0}^{4j+1}c_i&0 \end{array}\hspace{-1ex}\right);
\end{array}
$$
$$
C_{4(m+j)+6}\mathbf{C_{4j+2}}C_{4(m-j)} \equiv \left(\hspace{-1ex}\begin{array}{cc} t&\hspace{-1ex}(6j+7/2)a^2\\
1&0 \end{array}\hspace{-1ex}\right)\left(\hspace{-1ex}\begin{array}{cc} 0&\hspace{-1.5ex}\prod_{i=0}^{4j+2} c_i\\
\prod_{i=0}^{4j+1}c_i&0 \end{array}\hspace{-1ex}\right)\left(\hspace{-1ex}\begin{array}{cc} t&\hspace{-1ex}-(6j+13/2)a\\
1&0 \end{array}\hspace{-1ex}\right)\equiv \left(\hspace{-1ex}\begin{array}{cc} 0&\hspace{-1.5ex}\prod_{i=0}^{4j+3} c_i\\
\prod_{i=0}^{4j+2}c_i&0 \end{array}\hspace{-1ex}\right);
$$
$$
\begin{array}{rl}
C_{4(m+j)+7}\mathbf{C_{4j+3}}C_{4(m-j)-1} &\equiv \left(\hspace{-1ex}\begin{array}{cc} 12(j+1)t&\hspace{-1ex}(6j+13/2)a\\
1&0 \end{array}\hspace{-1ex}\right)\left(\hspace{-1ex}\begin{array}{cc} 0&\hspace{-1.5ex}\prod_{i=0}^{4j+3} c_i\\
\prod_{i=0}^{4j+2}c_i&0 \end{array}\hspace{-1ex}\right)\left(\hspace{-1ex}\begin{array}{cc} -12(j+1)t&\hspace{-1ex}-(6j+11/2)a\\
1&0 \end{array}\hspace{-1ex}\right)\\
&\equiv \left(\hspace{-1ex}\begin{array}{cc} 0&\hspace{-1.5ex}\prod_{i=0}^{4j+4} c_i\\
\prod_{i=0}^{4j+3}c_i&0 \end{array}\hspace{-1ex}\right).
\end{array}
$$
}

The last observation implies that, as soon as $c_l\equiv 0$ (mod $p$), the
entries of $\mathbf{C}_l$ with indices~11 and 12, as well as all the entries
of $\mathbf{C}_{l+1}$ are multiples of $p$. Now, if $p$ is of the form  $4l
+1$ then it divides $4m+5$ for all $m = l-1 + rp$ where $r$ is an arbitrary
integer. If on top of that $3\nmid p$ then either $p=4l+1 = 12i+5$ for some
integer $i$ and hence $p\mid c_{4i+1}$ or $p=12i+13$ and $p\mid c_{4i+3}$. In
both cases, we have that $p\mid c_j$ for an integer
$j\in[\frac{p-4}3,\frac{p-2}3]$ and hence the product of matrices
$C_{4rp+2p}C_{4rp+2p-1}\cdots C_{4rp+1}$ has all its entries divisible by
$p$. Similarly, if $p$ is of the form $4l+3$, then for any $m = 3l+1 + rp$
with $r\in\ZZ$ we have $p\mid 4m+5$. Also, if $3\nmid p$ then $p\mid c_j$ for
either $j = \frac{p+2}3$ or $j=\frac{p+4}3$. It is easy to derive then that
all the entries of the product $C_{4rp+4p}C_{4rp+4p-1}\cdots C_{4rp+2p+1}$
are divisible by $p$.

In a similar way we can deal with any divisor $p$ of $4m+3$. In that case we
have $m\equiv -3/4$ (mod $p$) and the following congruences modulo $p$ are
satisfied
$$
C_{4(m+l)+4} \equiv \left(\begin{array}{cc} t&(6l+5/2) a\\
1&0
\end{array}
\right);\quad
C_{4(m+l)+3} \equiv \left(\begin{array}{cc} (12l+6)t&(6l+7/2) a\\
1&0
\end{array}
\right);
$$
$$
C_{4(m+l)+2} \equiv \left(\begin{array}{cc} t&(6l+1/2) a^2\\
1&0
\end{array}
\right);\quad
C_{4(m+l)+1} \equiv \left(\begin{array}{cc} 12lr&(6l-1/2) a^2\\
1&0
\end{array}
\right).
$$
By induction by $l$, one can verify that
$$
\mathbf{D}_l:=\prod_{k=4m+1+l}^{4m+1-l} C_k \equiv \left(\begin{array}{cc} 0&\prod_{i=0}^l d_i\\
\prod_{i=0}^{l-1}d_i&0
\end{array}
\right)\; (\mathrm{mod}\;p),\quad \mbox{where}\; \left\{
\begin{array}{l}
c_{4i} = -(6i+1/2)a^2\\
c_{4i+1} = (6i+7/2)a\\
c_{4i+2} = -(6i+5/2)a\\
c_{4i+3} = (6i+11/2)a^2
\end{array}\right.
$$
Similarly as before, for $p$ of the form $4l+1$ such that $3\nmid p$ we have
that all the entries of the product $C_{4rp+4p}C_{4rp+4p-1}\cdots
C_{4rp+2p+1}$ are divisible by $p$. And for $p$ of the form $4l+3$, $3\nmid
p$, we have that all the entries of $C_{4rp+2p}C_{4rp+2p-1}\cdots C_{4rp+1}$
are divisible by $p$.

The upshot of the above arguments is that the product of matrices
$\prod_{i=k}^1 C_k$ can be split in at least $\lfloor \frac{k}{2p}\rfloor$
pieces that have all their entries divisible by $p$. Hence all the entries of
the whole product are divisible by $p^{\lfloor \frac{k}{2p}\rfloor}$. After
collecting the information for all primes $p\ge 5$ and taking into account
that $S_k = \prod_{i=4k}^1 C_kS_0$ we get~\eqref{lem5_eq}. \endproof

For simplicity, denote the right hand side of~\eqref{lem5_eq} by $g(k)$. We
estimate this function by something nicer. We have
$$
2^{\sum_{i=1}^\infty \frac{2k}{2^i}}\cdot 3^{\sum_{i=1}^\infty
\frac{2k}{3^i}}\cdot  \prod_{p\in\PP\atop p\ge 5} p^{\left\lfloor
\frac{2k}{p}\right\rfloor} \cdot \prod_{p\in \PP\atop p\ge 5}
p^{\sum_{i=2}^\infty \frac{2k}{p^i}} \ge (2k)!\ge \sqrt{4\pi k}
\left(\frac{2k}{e}\right)^{2k}.
$$
Therefore
\begin{equation}\label{eq42}
g(k)\ge \sqrt{4\pi k} (c_1 k)^{2k},\quad\mbox{where }\;
c_1=\frac{1}{\sqrt{3}e} \cdot \exp\left(-\sum_{p\in\PP\atop p\ge 5}
\frac{\ln p}{p(p-1)}\right) \approx 0.16948.
\end{equation}

Consider $p_k^*:=p_{4k+2}/\gcd(p_{4k+2}, q_{4k+2})$ and
$q_k^*:=q_{4k+2}/\gcd(p_{4k+2}, q_{4k+2})$. Definitely, they are both integer
numbers and Lemma~\ref{lem5} together with~\eqref{eq42} and~\eqref{eq11} give
us
$$
q_k^* < \frac{4(3t^2+a)\sqrt{k}}{\sqrt{\pi}}
\left(\frac{c_6}{c^2_1}\right)^k:= Q(k,t,a).
$$
Then with help of Lemma~\ref{lem7} we compute
$$
||q^*_k\xi|| \le |q^*_k\xi - p^*_k| = q^*_k\cdot \left|\xi -
\frac{p^*_k}{q^*_k}\right| < \frac{105\sqrt{3}e^2a^4
\sqrt{k}}{2\sqrt{\pi} t(t^2+2a)(3t^2+a)}\cdot
\left(\frac{9a^6(t^2+a)}{c^2_1e^2 t^4(t^2+2a)^2}\right)^k =:
R(k,t,a).
$$
Notice that by the definition of $\tau_1 , \tau_2 , c_2 $ and $c_3 $
in Theorem~\ref{th3}, we have $Q(k,t,a) = \tau_1 \sqrt{k} c_2 ^k$
and $R(k,t,a)=\tau_2
 \sqrt{k} c_3
^{-k}$.


%

{\bf Proof of Theorem~\ref{th3}.} Consider an arbitrary $q \ge
\frac{1}{2R(1,t,a)}$. Notice that since $c_3
>e$ then the sequence
$c_3 ^k/\sqrt{k}$ is strictly increasing for all $k\ge 1$.
Therefore, there exists a unique $k\ge 2$ such that $R(k,t,a)<
\frac{1}{2q}\le R(k-1,t,a)$. Let $p\in \ZZ$ be such that $||q\xi|| =
|q\xi-p|$. Since two vectors $(p_k^*, q_k^*)$ and $(p_{k+1}^*,
q_{k+1}^*)$ are linearly independent, at least one of them must be
linearly independent with $(p,q)$. Suppose that is $(p_k^*, q_k^*)$.
Then we estimate the absolute value of the following determinant:
$$
1\le \left| \begin{array}{cc}
q&q_k^*\\
p&p_k^*
\end{array}\right| \le \left| \begin{array}{cc}
q&q_k^*\\
p-q\xi& p_k^* - q_k^*\xi
\end{array}\right| \le q R(k,t,a) + ||q\xi||Q(k,t,a).
$$
Since $qR(k,t,a) < \frac12$, we must have $||q\xi|| \ge
(2Q(k,t,a))^{-1}$. Analogously, if $(p,q)$ is linearly independent
with $(p_{k+1}^*, q_{k+1}^*)$, we have $||q\xi||\ge
(2Q(k+1,t,a))^{-1}$. The latter lower bound is weaker. Now, we need
to rewrite the right hand side of the inequality in terms of $q$
rather than $k$.

From $\frac{1}{2q}\le R(k-1,t,a)$ we have that
$$
\frac{c_3 ^{k-1}}{2\tau_2 \sqrt{k-1}} \le q.
$$
We show that the last inequality implies
\begin{equation}\label{eq12}
k-1\le \frac{\log(2\tau_2
 q) + \log\log(2\tau_2
  q)}{\log c_3
}.
\end{equation}
Indeed, if we take $\kappa$ such that $\kappa\log c_3
 = \log(2\tau_2
q) + \log\log(2\tau_2
 q)$ then $c_3
^{\kappa}/ \sqrt{\kappa} \ge 2\tau_2 q$ is equivalent to
$$
\log\log (2\tau_2 q) \ge \frac12 \log\left(\frac{\log (2\tau_2 q) +
\log\log (2\tau_2
 q)}{\log c_3
}\right).
$$
By taking exponents of both sides and after simplifications, we
derive $\log c_3 \cdot \log(2\tau_2 q) \ge 1 + \frac{\log\log
(2\tau_2 q)}{\log (2\tau_2
 q)}$. In view of $\log c_3
 > 1$, the last inequality is
satisfied for all $2\tau_2
 q>1$.

Substitute that into $||q\xi||\ge (2Q(k+1,t,a))^{-1}$ and get
$$
||q\xi|| \ge \frac{1}{2\tau_1
 \sqrt{k+1}c_2
^{k+1}} \stackrel{\eqref{eq12}}{\ge} \frac{(\log c_3
)^{1/2}}{2\sqrt{3}\tau_1 c^2_2 (\log(2\tau_2 q) + \log\log(2\tau_2
q))^{1/2} (2\tau_2 q)^{\frac{\log c_2 }{\log c_3 }}\cdot \log
(2\tau_2 q)^{\frac{\log c_2 }{\log c_3 }}}.
$$$$
\ge \frac{(\log c_3 )^{1/2}}{6\tau_1
 c_2
^2(2\tau_2 )^\frac{\log c_2 }{\log c_3 }} \cdot q^{-\frac{\log c_2
}{\log c_3 }} \cdot \log (2\tau_2 q)^{-\frac{\log c_2 }{\log c_3
}-\frac12}.
$$
\endproof

\section{Very good rational approximations to cubic irrationals}\label{sec4}

In this section we will look for the best possible rational
approximations which can be derived from the continued
fraction~\eqref{no5} and then prove Theorem~\ref{th4}.

Given $n\in\NN$, define $d_2(n)$ to be the largest odd divisor of
$n$. For convenience, in the remaining part of this section we
always have $a=1$. We apply the following continued fraction
transformations:
\begin{itemize}
\item divide $a_1$ and $\beta_1$ by $d_2(\beta_1)$ and then multiply
    $a_{2}$ and $\beta_{3}$ by $d_2(\beta_1)$;
\item repeat the similar transformations for indices $i = 3, \ldots, n$:
    divide the values $a_i$ and $\beta_i$ of a newly achieved continued
    fraction by $d_2(\beta_i)$ and then multiply $a_{i+1}$ and
    $\beta_{i+2}$ by $d_2(\beta_i)$.
\end{itemize}
By Lemma~\ref{lem8}, they do not change the sequence of convergents
of the continued fraction.  Let $a_i^*$ and $\beta_i^*$ be the
coefficients of the continued fraction after performing these
transformations. One can verify that for all $i,k\in \ZZ_{\ge 0}$,
$1\le i, 4k+4\le n$ they satisfy
\begin{equation}\label{ast1}
\beta^*_i = \frac{\beta_i}{d_2(\beta_i)};\; a^*_{4k+1} = \frac{a_{4k+1}}{d_2(3k+1)} \cdot \frac{\prod_{j=0}^{k-1}(6j+1)(6j+5)}{\prod_{j=0}^{k-1} d_2(3j+1)d_2(3j+2)}.
\end{equation}
\begin{equation}\label{ast2}
a^*_{4k+2} = \frac{a_{4k+2}d_2(3k+1)}{(6k+1)} \cdot
\frac{\prod_{j=0}^{k-1}
d_2(3j+1)d_2(3j+2)}{\prod_{j=0}^{k-1}(6j+1)(6j+5)}.
\end{equation}
\begin{equation}\label{ast3}
a^*_{4k+3} = a_{4k+3}(6k+1) \cdot
\frac{\prod_{j=0}^{k-1}(6j+1)(6j+5)}{\prod_{j=0}^{k}
d_2(3j+1)d_2(3j+2)};\;
\end{equation}
\begin{equation}\label{ast4}
a^*_{4k+4} = a_{4k+4} \frac{\prod_{j=0}^{k}
d_2(3j+1)d_2(3j+2)}{\prod_{j=0}^{k}(6j+1)(6j+5)}.
\end{equation}

Let $t = t_k := lcm\{a\in \ZZ\;:\; 1\le a\le 6k+1; a\equiv \pm
1\;(\mathrm{mod}\; 6)\}$.
\begin{lemma}\label{lem10}
With $t=t_k$, all values $a^*_i$ for $1\le i\le 4k+3$ are integer.
\end{lemma}

\proof Since each $a_i$ is an integer multiple of $t$, it is
sufficient to show that all the values~\eqref{ast1}--\eqref{ast4}
still remain integer if $a_i$'s are replaced by $t$. Also notice
that for $m\in \ZZ, m\ge 0$, the number $d_2(3m+1)$ appears in the
denominator of $a^*_{4m+1}$ for the first time and then occurs in
the fractions for all $a^*_i$, $i\ge 4m+1$. Analogously, $d_2(3m+2)$
appears in the denominator of $a^*_{4m+3}$, $6m+1$ appears in
$a_{4m+2}$ and $6m+5$ appears in $a_{4m+4}$ for the first time.

Let $p>3$ be integer which is coprime with 6. First, assume that
$p=6m+1$ for some integer $m$. We investigate, for what indices $i$
do multiples of $p$ appear in the denominators of $a_i^*$ for the
first time.
\begin{itemize}
\item $(6l+1)(6m+1)$ appears at index $4(6ml+m+l)+2$;
\item $(6l+5)(6m+1)$ appears at index $4(6ml+5m+l)+4$;
\item $d_2((3l+1)(6m+1))$ appears at index $4(6ml+2m+l)+1$;
\item $d_2((3l+2)(6m+1))$ appears at index $4(6ml+4m+l)+3$.
\end{itemize}
Notice that the first two and the last two factors always appear at
the opposite sides of the fractions~\eqref{ast1}--\eqref{ast4}.
Obviously, $6ml+m+l < 6ml+2m+l < 6ml+4m+l < 6ml + 5m+ l < 6m(l+1) +
m + (l+1)$ for all integers $m>0, l\ge 0$. Therefore, we conclude
that the amount of multiples of $p$ in the numerators and
denominators of each fraction~\eqref{ast1}--\eqref{ast4} does not
differ by more than 1.

The case $p=6m+5$ is considered analogously. We have
\begin{itemize}
\item $(6l+1)(6m+5)$ appears at index $4(6ml+m+5l)+4$;
\item $(6l+5)(6m+5)$ appears at index $4(6ml+5m+5l+4)+2$;
\item $d_2((3l+1)(6m+5))$ appears at index $4(6ml+2m+5l+1)+3$;
\item $d_2((3l+2)(6m+5))$ appears at index $4(6ml+4m+5l+3)+1$.
\end{itemize}
Then one can easily check that again the amount of multiples of $p$
in the numerators and denominators of each
fraction~\eqref{ast1}--\eqref{ast4} do not differ by more than 1.

Consider a prime number $p>3$. Let $d = d_{p,k}$ be the largest
power of $p$ such that $p^d\le 6k+1$. We have shown that for all
$i\in\NN$, the number of multiples of $p$ in the denominators
of~\eqref{ast1}--\eqref{ast4} do not exceed the number of multiples
of $p$ in the numerators by one. The same observation is true for
multiples of $p^2, p^3, \ldots, p^d$. Finally, there are no
multiples of $p^{d+1}$ in the fractions for $a^*_i$ with $i\le
4k+3$. Then one concludes that, after all cancelations, the power of
$p$ in the denominators of~\eqref{ast1}--\eqref{ast4} does not
exceed $d$.

Finally, the fact that $t_k = \prod_{p\in \PP_{>3}} p^{d_{p,k}}$
concludes the proof.
\endproof

Notice that
$$
t_k = \frac{e^{\psi(6k+1)}}{2^{\big\lfloor\frac{\log(6k+1)}{\log
2}\big\rfloor} 3^{\big\lfloor\frac{\log(6k+1)}{\log 3}\big\rfloor}},
$$
where $\psi(n)$ is the second Chebyshev function. Then, one can use $\psi(n)
\approx n$ to provide a good estimate for $t_k$. In fact, for all
$\epsilon>0$ there exists $k_0\in \ZZ$ such that for all $k>k_0$,
\begin{equation}\label{eq23}
\frac{e^{(1-\epsilon)(6k+1)}}{(6k+1)^2}< t_k <
\frac{6e^{(1+\epsilon)(6k+1)}}{(6k+1)^2}.
\end{equation}
For small values of $\epsilon$ the value of $k_0$ may become big and
hard to locate. To make the bounds precise and effective (which is
needed for Theorem~\ref{th4} and Conjecture~B) we use $0.96x\le
\psi(x)\le 1.04x$ from~\cite{ros_sch_62} to get:
\begin{equation}\label{eq23_1}
\frac{e^{0.96(6k+1)}}{(6k+1)^2}< t_k <
\frac{6e^{1.04(6k+1)}}{(6k+1)^2}.
\end{equation}
The first inequality is satisfied for $k\ge 2$.

We have that for all $i\le 4k+3$ the values of $a^*_i$ and $\beta_i$
are positive integers, therefore $q_i>q_{i-1}$. Moreover,
$$
a^*_i q_{i-1}<q_i = a^*_i q_{i-1} + \beta^*_i q_{i-2} < (a^*_i +
\beta^*_i)q_{i-1}.
$$
For even values of $i$ we have $\beta^*_i = 1\le a^*_i$, while
$\beta^*_{4m+1} \le 3m+1 < 4m+1 \le a^*_{4m+1}$. Analogously,
$\beta^*_{4m+3} \le 3m+2 < 4m+3 \le a^*_{4m+3}$. Therefore we can
simplify the upper bound for $q_i$:
\begin{equation}\label{eq21}
q_i < 2a_i^* q_{i-1}.
\end{equation}


The inequality $(a_i^*+\beta_i^*)q_{i-1}>q_i$ for $i=m\le k$ implies
$$
q_{4m+3} <
\left(\frac{3(4m+3)}{(6m+5)}t_k(t_k^2+2)\prod_{j=0}^m
\frac{4(3j+1)(3j+2)}{d_2(3j+1)d_2(3j+2)} +
\frac{6m+4}{d_2(3m+2)}\right) q_{4m+2}.
$$
Denote the product of $4(3j+1)(3j+2)/d_2(3j+1)d_2(3j+2)$ in the formula above
by $P_m$. Notice that it is bigger than $4^m$ which is larger than $6m+4$ for
$m\ge 3$. Then we transform the inequality to a simpler form
\begin{equation}\label{eq19}
q_{4m+3} < 8P_m t_k^3 q_{4m+2}.
\end{equation}

For the opposite estimate between $q_{4m+3}$ and $q_{4m+2}$, we provide a
lower bound on $a^*_{4m+3}$:
$$
a_{4m+3}^*= \frac{3(4m+3)}{6m+5}t_k(t_k^2+2) P_m \prod_{j=0}^m \frac{(6j+1)(6j+5)}{(6j+2)(6j+4)}.
$$
As $m$ grows, the product in this expression decreases and converges
to $0.577...$, therefore we get an estimate
\begin{equation}\label{eq18}
q_{4m+3} > a_{4m+3}^* q_{4m+2} > P_m t_k^3 q_{4m+2}.
\end{equation}

As the next step, we estimate $||q_{4m+2} \xi||$ for all $m<k$. By
Lemma~\ref{lem7}, we have $$\left| \xi -
\frac{p_{4m+2}}{q_{4m+2}}\right| < 2\left|\frac{p_{4m+2}}{q_{4m+2}}
- \frac{p_{4m+6}}{q_{4m+6}}\right|.
$$
We also derive
$$
\max_{4m+2\le i\le 4m+5}\left|\frac{p_i}{q_i} -
\frac{p_{i+1}}{q_{i+1}}\right| \le \frac{\left| \prod_{i=-1}^{4m+4}
\det C^*_i\right|}{q_{4m+2}q_{4m+3}}\le \prod_{i=1}^{4m+6}
\frac{\beta_i}{d_2(\beta_i)} \cdot \frac{1}{q_{4m+2}q_{4m+3}} \le
\frac{2(3m+4)P_m}{d_2(3m+4)q_{4m+2}q_{4m+3}},
$$
where the matrices $C_i^*$ are defined in~\eqref{eq14} but with the
new partial quotients $a_i^*, \beta_i^*$ instead of $a_i, \beta_i$.
Together with~\eqref{eq18}, the last two inequalities yield
\begin{equation}\label{eq20}
|q_{4m+2}\xi - p_{4m+2}| < \frac{8(3m+4)}{d_2(3m+4) t_k^3 q_{4m+2}}.
\end{equation}

\begin{lemma}\label{lem16}
The value $P_m$ satisfies
$$
P_m \le (3m+2) 2^{4m+4}
$$
\end{lemma}

\proof We can compute $P_m$ by counting the number of even integers
between 1 and $3m+2$ that are not multiples of 3, then count the
number of integers in the same set, that are multiples of 4,
multiples of 8, etc. We stop when we reach the maximal power of two
not exceeding $3m+2$.

The number of multiples of $2^j$ among all numbers between 1 and
$3m+2$ is $\big\lfloor \frac{3m+2}{2^j}\big\rfloor$. Among them,
$\big\lfloor\frac{m}{2^j}\big\rfloor$ numbers are multiples of 3.
For $j=1$ we have
$$
\left\lfloor \frac{3m+2}{2^j}\right\rfloor -
\left\lfloor\frac{m}{2^j}\right\rfloor = m+1
$$
and for bigger values of $j$,
$$
\left\lfloor \frac{3m+2}{2^j}\right\rfloor
- \left\lfloor\frac{m}{2^j}\right\rfloor < \frac{m+1}{2^{j-1}} + 1.
$$
Combining these inequalities for all $j \in\{1,2,\ldots,
\log_2(3m+2)\}$, we derive
$$
\frac{P_m}{4^{m+1}} < 2^{(m+1)\sum_{j=0}^{m-1} 2^{-j} +
\log_2(3m+2)} < (3m+2)2^{2m+2}.
$$

\endproof

In the next lemma, we show that $p_{4m+2}$ and $q_{4m+2}$ are both
multiples of a large power of~2. Therefore we can cancel it from
$p_{4m+2}$ and $q_{4m+2}$ and hence significantly improve the
inequality~\eqref{eq20}.

\begin{lemma}\label{lem9}
Let $t$ be odd and $a=1$. Then all the entries of the matrix
$A_{4n+3}A_{4n+2}A_{4n+1}A_{4n}$ are multiples of $2^7$, and all the
entries of the matrix $A_{4n+7}\cdots A_{4n+1}A_{4n}$ are multiples
of $2^{15}$.
\end{lemma}

\proof

Let $m$ be odd. Then both $6m+4$ and $6m+8$  are multiples of 2 but
not 4. Then we use~\eqref{eq14} to compute
$$
C_{4m+2}C_{4m+1} \equiv \left(\begin{array}{cc}
t^4+2t^2 + 3 & 2t\\
t(t^2+2)& 2
\end{array}\right) \equiv \left(\begin{array}{cc}
2 & 2\\
1\mbox{ or }3& 2
\end{array}\right) \; (\mathrm{mod}\; 4).
$$
Analogously,
$$
C_{4m+4}C_{4m+3} \equiv \left(\begin{array}{cc}
3t^2 + 1 & 2t\\
3t& 2
\end{array}\right) \equiv \left(\begin{array}{cc}
0 & 2\\
1\mbox{ or }3& 2
\end{array}\right) \; (\mathrm{mod}\; 4).
$$
After multiplying the last two equations we derive
\begin{equation}\label{eq15}
A_m =
\left(\begin{array}{cc}
2u_m & 4v_m\\
4w_m& 2t_m
\end{array}\right),
\end{equation}
where $u_m,v_m,t_m$ are odd integers and $w_m$ is an integer (odd or even).

If $m=4n$, we have that $6m+4$ is a multiple of 4 but not 8, and $6m+8$ is a
multiple of $8$. Let $2^{\delta_m}$ be the largest power of 2 that divides
$6m+8$. Then we compute
$$
C_{4m+2}C_{4m+1} = \left(\begin{array}{cc}
8e_{m1} & 4e_{m2}\\
e_{m3}& 4e_{m4}
\end{array}\right); \quad C_{4m+4}C_{4m+3} = \left(\begin{array}{cc}
2e_{m5} & 2^{\delta_m} e_{m6}\\
e_{m7}& 2^{\delta_m}e_{m8}
\end{array}\right),
$$
where $e_{m2}, \ldots e_{m8}$ are some odd integers and $e_{m1}$ shares the
parity with $n$. Therefore, in this case we get
\begin{equation}\label{eq16}
A_m = \left(\begin{array}{cc}
2^{\tau_m}u_m & 2^3v_m\\
2^3w_m& 2^2t_m
\end{array}\right),\quad \tau_m = \left\{ \begin{array}{rl}
3&\mbox{if } \delta_m = 3;\\
5&\mbox{if }\delta_m = 4;\\
4&\mbox{if }\delta_n >4
\end{array}\right.
\end{equation}
where $v_m,w_m,t_m$ are odd integers, and $u_m$ is integer (not necessarily
odd).

Consider the last case $m = 4n+2$. Then we have that $6m+8$ is a multiple of
4 but not 8, and $6m+4$ is a multiple of $8$. Let $2^{\delta_m}$ be the
largest power of 2 that divides $6m+4$. We compute
$$
C_{4m+2}C_{4m+1} = \left(\begin{array}{cc}
4e_{m1} & 2^{\delta_m}e_{m2}\\
e_{m3}& 2^{\delta_m}e_{m4}
\end{array}\right); \; \quad C_{4m+4}C_{4m+3} = \left(\begin{array}{cc}
2e_{m5} & 4 e_{m6}\\
e_{m7}& 4e_{m8}
\end{array}\right),
$$
where $e_{m1}, \ldots, e_{m8}$ are all odd integer numbers. Combined, these
formulae infer
\begin{equation}\label{eq17}
A_m = \left(\begin{array}{cc}
4u_m & 2^{\delta_m+1}v_m\\
8w_m& 2^{\delta_m}t_m
\end{array}\right),
\end{equation}
where $u_m,v_m$ and $t_m$ are odd integers and $w_m$ is any integer.

As a final step, we notice that $\delta_{4n}=3$ if and only if
$\delta_{4n+2}>3$. Then we combine~\eqref{eq15},\eqref{eq16}
and~\eqref{eq17} to get
$$
A_{4n+3}A_{4n+2}A_{4n+1}A_{4n} = \left(\begin{array}{cc}
\!\!2^{\lambda_n}U_n & 2^8V_n\!\!\\
\!\!2^8W_n& 2^{\theta_n}T_n\!\!
\end{array}\right),\; \lambda_n=\left\{\begin{array}{rl}
7&\mbox{if } \delta_{4n}=3;\\
8&\mbox{if } \delta_{4n}>3.
\end{array}\right.;\; \theta_n=\left\{\begin{array}{rl}
7&\mbox{if } \delta_{4n+2}=3;\\
8&\mbox{if } \delta_{4n+2}>3.
\end{array}\right.
$$
where $U_n, V_n, W_n, T_n$ are some integer numbers. That verifies
the first claim of the lemma. For the second claim, we need to
multiply two consecutive blocks $A_{4n+7}A_{4n+6}A_{4n+5}A_{4n+4}$
and $A_{4n+3}A_{4n+2}A_{4n+1}A_{4n}$ and notice that the consecutive
values of $\lambda_n$ and $\theta_n$ alternate.

\endproof

Lemma~\ref{lem9} implies that for all $n\in\ZZ$, $2^{15n}$ divides $
p_{32n+1}, q_{32n+1}, p_{32n+2}$ and $q_{32n+2}$. Finally, we get
that $q_{32n+3} = a_{32n+3} q_{32n+2} + \beta_{32n+3}q_{32n+1}$ is
also a multiple of $2^{15n}$. The same condition is satisfied for
$p_{32n+3}$.

In further discussion, we specify $k$ to be divisible by 8, i.e. $k = 8k_0$.
Then $q_{4k+2}$ and $p_{4k+2}$ are all multiples of $2^{15k_0}$, i.e.
$q_{4k+2} = 2^{15k_0} q^*_{4k+2}$. Then for all $m<k_0$,~\eqref{eq20} implies
\begin{equation}\label{eq22}
||q^*_{32m+2}\xi|| < \frac{8(24m+4)}{d_2(24m+4)t_k^3
2^{30m}q^*_{32m+2}}= \frac{32}{4^{15m}t_k^3 q^*_{32m+2}}.
\end{equation}

{\bf Remark.} Direct computations suggest that for infinitely many values of
$m$ the powers of 2 which divide both $p_{32m+2}$ and $q_{32m+2}$ are around
$2^{16m}$. Therefore, we believe that the power of $4$ in the denominator
of~\eqref{eq22} can be replaced by $4^{16m}$. However, we were not able to
prove it.


{\bf Proof of Theorem~\ref{th4}}. Fix a small value $\epsilon>0$. In view
of~\eqref{eq23}, consider $k_0$ large enough, so that $t_k <
e^{48(1+\epsilon)k_0}$ and $4^{\epsilon(1-\epsilon)k_0} > 32\cdot 3^4$.
Choose any $m$ in the range $(1-\epsilon)k_0< m < k_0$. Then for all $k_0$
large enough one has
$$
4^{-15(1-\epsilon)m} < e^{-48(1+\epsilon)k_0 \cdot
\frac{15\ln 4 \cdot (1-\epsilon)^2}{48(1+\epsilon)}} <
t_k^{-\frac{15\ln 2\cdot (1-\epsilon)^2}{24(1+\epsilon)}}.
$$
Therefore~\eqref{eq22} gives $||q^*_{32m+2}\xi|| < 32\cdot 3^4\cdot
4^{-15\epsilon m} (3t_k)^{-3-\frac{15\ln2\cdot
(1-\epsilon)^2}{24(1+\epsilon)}} (q^*_{32m+2})^{-1}$. In other
words, for all $\tau < 3 + \frac{15\ln 2}{24}\approx 3.4332...$
there exists $\epsilon>0$ such that for $k_0$ large enough and
$(1-\epsilon)k_0 < m< k_0$ one has
$$
||q^*_{32m+2}\xi|| < \frac{1}{4^{14\epsilon m}\cdot (3t_k)^\tau
q^*_{32m+2}}.
$$

Next, we estimate $q^*_{32m+2}$. By~\eqref{eq21}, $q_{32m+2} <
2^{32m+2} \prod_{i=1}^{32m+2} a_i^*$. Then we
use~\eqref{ast1}--~\eqref{ast4} to get
$$
q_{32m+2} \le \frac{2^{32m+2} \prod_{j=1}^{32m+2} a_j}{(48m+1)
\prod_{j=0}^{8m-1}(6j+1)(6j+5)}\stackrel{\eqref{no5}}{<}
\frac{2^{32m+2}3^{16m+1}t_k^{32m+2}(t_k^2+2)^{8m}
(32m+1)!!}{(48m+1)\prod_{j=0}^{8m-1}(6j+1)(6j+5)} .
$$
Since $t_k>8m$, $4j+1 < \frac23 (6j+5)$ and $4j+3 < \frac23 (6j+7)$, there
exists an effectively computable constant $C$ such that the last expression
is bounded from above by $C\cdot 2^{48m}t_k^{48m+2}(32m+1)$. Finally, we use
$t_k<e^{48(1+\epsilon)k_0}$, the fact that $q_{32m+2} = 2^{15m}q^*_{32m+2}$
and $(1-\epsilon)k_0<m<k_0$ to get that for $k_0$ large enough,
\begin{equation}\label{eq24}
q^*_{32m+2} < \frac{C\cdot
2^{48m}\left(e^{\frac{48(1+\epsilon)m}{1-\epsilon}}\right)^{48m+2}}{2^{15m}}
< e^{\frac{48^2(1+2\epsilon)}{1-\epsilon}m^2}.
\end{equation}
The last inequality yields
$$
m > \sqrt\frac{(1-\epsilon)\ln
(q^*_{32m+2})}{48^2(1+2\epsilon)}\quad\Longrightarrow\quad
4^{14\epsilon m}
> e^{\sqrt{\frac{49\epsilon^2(1-\epsilon)}{24^2(1+2\epsilon)}\ln
(q^*_{32m+2})}}.
$$

Choose $k_0$ to be large enough so that all the inequalities in the
arguments above take place and on top of that there are at least
$N_0$ values $m$ in the range $(1-\epsilon)k_0 < m< k_0$. Then the
height of a cubic irrational $\xi = x(t_k,1)$ from~\eqref{no5} is
$3t_k$ and therefore the inequality~\eqref{eq_th3} is satisfied for
all $q = q^*_{32m+2}$ and $c =
\frac{7\epsilon}{24}\sqrt{\frac{1-\epsilon}{1+2\epsilon}}$.

\endproof

In the end of this section we provide lower bounds for $||q_n\xi||$
that support Conjecture~B. Some of the arguments here will be
heuristic or numerical. First of all, notice that among all
continued fractions~\textnumero 1--6 only~\eqref{no5} has the
partial quotients which are cubic in $t$. The others are either
linear or quadratic. Therefore, for $\xi = x(t,a)$ from~\eqref{no5}
and large $t$, the expression $||q_{4k+2}\xi||$ will more likely
provide the smallest value in terms of the denominator $q_{4k+2}$
among other convergents that come from those continued fractions. We
also assume that $a=1$ because bigger absolute values of $a$ give
bigger determinants of matrices $C_n$ and hence bigger differences
$\left|\frac{p_n}{q_n} - \frac{p_{n+1}}{q_{n+1}}\right|$. Finally,
we assume that $t=t_k$ because in that case, as was shown above, the
continued fraction can be transformed in such a way that all the
values $\beta_i$, $i\le 4k+2$ can be significantly reduced to
$\beta^*_i = \frac{\beta_i}{d_2(\beta_i)}$ and the numerators and
denominators $(p_{4k+2}, q_{4k+2})$ of the convergents are divisible
by a large power of two. On top of that, numeric computations
suggest that such values of $t$ provide the largest possible partial
quotients $a_{4k+2}$.

From Lemma~\ref{lem10} we have that
$$
t_k \cdot \frac{\prod_{j=0}^{k-1}(6j+1)(6j+5)}{\prod_{j=0}^{k}
d_2(3j+1)d_2(3j+2)}
$$
is integer and therefore $q_{4k+4}>q_{4k+3}> a^*_{4k+3} q_{4k+2} \ge
(t_k^2+2)q_{4k+2}$. Now we estimate
$$
\max_{4k+3\le i\le 4k+5} \left|\frac{p_i}{q_i}-
\frac{p_{i+1}}{q_{i+1}}\right| \le\left| \frac{p_{4k+2}}{q_{4k+2}} -
\frac{p_{4k+3}}{q_{4k+3}}\right| \cdot |\det C_{4k+2}\det
C_{4k+3}\det C_{4k+4}| \cdot \frac{q_{4k+2}}{q_{4k+4}}
$$
In view of~\eqref{no5}, for $k\ge 2$ the last expression is smaller
than
$$
\left|\frac{p_{4k+2}}{q_{4k+2}} - \frac{p_{4k+3}}{q_{4k+3}}\right|
\cdot \frac{(6k+5)\cdot 2(3k+4)(6k+7)}{t_k^2+2} < \frac14
\left|\frac{p_{4k+2}}{q_{4k+2}} - \frac{p_{4k+3}}{q_{4k+3}}\right|
$$
The last inequality follows from~\eqref{eq23_1}. Finally, we use Lemma~\ref{lem7} to conclude that
$$
\left|\xi - \frac{p_{4k+2}}{q_{4k+2}}\right| \ge
\left|\frac{p_{4k+2}}{q_{4k+2}} - \frac{p_{4k+6}}{q_{4k+6}}\right|
\ge \left|\frac{p_{4k+2}}{q_{4k+2}} -
\frac{p_{4k+3}}{q_{4k+3}}\right| - \left|\frac{p_{4k+3}}{q_{4k+3}} -
\frac{p_{4k+6}}{q_{4k+6}}\right| \ge \frac14
\left|\frac{p_{4k+2}}{q_{4k+2}} - \frac{p_{4k+3}}{q_{4k+3}}\right|.
$$
The right hand side is bounded from below by
$$
\frac14
\left|\frac{p_{4k+2}}{q_{4k+2}} - \frac{p_{4k+3}}{q_{4k+3}}\right| = \frac{|\prod_{i=-1}^{4k+2}\det C_i^*|}{4q_{4k+2}q_{4k+3}} \ge \frac{P_k}{4q_{4k+2}q_{4k+3}}.
$$
We apply~\eqref{eq19} to get the estimate
\begin{equation}\label{eq25}
||q_{4k+2}\xi|| \ge \frac{P_k}{4\cdot 8 P_k t_k^3 q_{4k+2}} \gg
\frac{1}{t_k^3 q_{4k+2}}.
\end{equation}

Since $|p_{4k+2}q_{4k+3}-p_{4k+3}q_{4k+2}| = \beta_{4k+3}^*P_k$ is a power of
2, we have that $\gcd(p_{4k+2}, q_{4k+2})$ is also a power of two.
Computations suggest that
\begin{equation}\label{eq26}
\gcd(p_{4k+2}, q_{4k+2}) \asymp 2^{2k},
\end{equation}
but we only manage to prove a weaker inequality $\gcd(p_{32k+2},
q_{32k+2}) \ll k2^{17k}$, by using Lemmata~\ref{lem16}
and~\ref{lem9}.


Denote by $q^*_{4k+2}$ the ratio $\frac{q_{4k+2}}{\gcd(p_{4k+2},
q_{4k+2})}$. We will provide the lower bounds for
$||q^*_{4k+2}\xi||$ based on the numerical evidence (stronger) and
based on the rigorous proof (weaker). Numerically, from~\eqref{eq25}
and~\eqref{eq26} we have
$$
||q^*_{4k+2}\xi|| \gg \frac{1}{2^{4k}t_k^3 q^*_{4k+2}}.
$$
Then,~\eqref{eq23_1} gives that $2^{4k} = e^{0.96(6k)\cdot
\frac{2\ln2}{0.96\cdot 3}}\ll k^2 t_k^{\frac{2\ln2}{2.88}}$. Whence,
\begin{equation}\label{lowerb_num}
||q^*_{4k+2}\xi|| \gg \frac{1}{k^2\cdot
t_k^{3+\frac{2\ln2}{2.88}}q^*_{4k+2}}.
\end{equation}

Based on these arguments (not completely rigorous) and the
assumption that $p^*_{4k+2}/q^*_{4k+2}$ is a convergent of $\xi$ of
index comparable to $k$, we formulate our Conjecture~B.

Using~\eqref{eq23} instead of~\eqref{eq23_1} will result in a
slightly better power $3 + \frac23\ln 2 (1+\epsilon)$ of $t_k$ that
supports Conjecture~A. However, for very small $\epsilon$ the
inequality will only be valid for huge values of $k$. That makes the
task of computing the precise value of the constant $C$, as in
Conjecture~B, harder but theoretically possible.


\section{Continued fractions for all cubic irrationals}\label{sec6}

In this section, for any real cubic irrational $\xi$ we construct a
M\"obius transform $\mu$ and a continued fraction of the
form~\textnumero 4 that converges to $\mu(\xi)$. This will confirm
Theorem~\ref{th2}.

Consider an arbitrary cubic algebraic $\xi\in \RR$ with the minimal
polynomial $P_\xi(x) = b_3x^3+b_2x^2+b_1x+b_0 \in \ZZ[x]$. One can
easily check that the minimal polynomial of $\eta_*:=
(3b_3\xi+b_2)^{-1}$ is of the form $B_3y^3 + B_2y^2 + 1=0$, where
$B_3$ and $B_2$ are integer. Finally, the minimal polynomial of
$\eta:= B_3\eta_*$ is $y^3 + B_2y^2 + B_3^2=0$, which is of the
form~\textnumero 4 with $t=-B_2$ and $a = -B_3^2$. That immediately
implies that for at least one root $\xi_*$ of $P_\xi(x)$, the number
$\eta = \frac{B_3}{b_3\xi_*+b_2}$ is linked with the continued
fraction expansion~\eqref{no4}. However, firstly we want the
continued fraction~\eqref{no4} to converge, but that is not
guaranteed for an arbitrary $\eta$. Secondly, we want to cover all
real roots of $P_\xi(x)$, not just one of them. In this section we
will resolve these two problems and provide a (generalised)
continued fraction expansion in a closed form for all cubic
irrationals $\xi\in\RR$.

Let $\xi$ be a real root of the cubic polynomial $P(x)$. Consider
$\eta\in\RR$ such that $\xi = \frac{u\eta+v}{s\eta+w}$. We will find
the conditions on $u,v,s,w\in\ZZ$ such that the coefficient at $y$
of the minimal polynomial $P_\eta(y)$ of $\eta$ equals zero. The
minimal polynomial $P_\eta$ is
\begin{equation}\label{eq29}
b_3(uy+v)^3 + b_2(uy+v)^2(sy+w) + b_1(uy+v)(sy+w)^2 + b_0(sy+w)^3 =
0.
\end{equation}
That is a cubic polynomial in $y$ and its coefficient at $y$ is
$$
3b_3uv^2 + 2b_2uvw+b_2v^2s + b_1uw^2 + 2b_1vsw+3b_0sw^2.
$$
Equating this coefficient to zero gives
$$
u(3b_3v^2+2b_2vw+b_1w^2) = -s(3b_0w^2 + 2b_1vw+b_2v^2).
$$
The solutions of this equation for arbitrary $v,w\in\ZZ$ are
\begin{equation}\label{eq28}
\begin{array}{rl}
s\!\!\! &= 3b_3v^2 + 2b_2vw+b_1w^2 =
w^2P_\xi'\left(\frac{v}{w}\right),\\[2ex]
u\!\!\! &= -v^2 \frac{d}{dx}(x^3P_\xi(1/x))\big|_{x = w/v} =
vwP_\xi'\left(\frac{v}{w}\right) -
3w^2P_\xi\left(\frac{v}{w}\right).
\end{array}
\end{equation}
To make the notation shorter, we will write $P$ instead of $P(v/w)$
and $P'$ instead of $P'(v/w)$.

With $u$ and $s$ as in~\eqref{eq28}, the free coefficient of~\eqref{eq29} is
obviously $w^3P$. Now, we compute the coefficient of this polynomial at
$y^2$: {\footnotesize$$ 3b_3v(vwP'-3w^2P)^2+
2b_2w^2vP'(vtP'-3w^2P)+b_2w(vwP'-3w^2P)^2 +
2b_1w^3P'(vwP'-3w^2P)+b_1vw^4(P')^2+3b_0w^5(P')^2
$$}\vspace{-3ex}
$$
= (3w^5P)(P')^2 - (6w^5P')PP' + 9(3b_3v + b_2w)w^4P^2
$$
\begin{equation}\label{eq39}
= 3w^3P((3b_3b_1 - b_2^2)v^2 +
(9b_3b_0-b_2b_1)vw+(3b_2b_0-b_1^2)w^2) =: 3w^5 PR(v/w).
\end{equation}
We verify that $R$ is not the constant zero. Indeed, if that is the
case, we must have $3b_3b_1=b_2^2, 9b_3b_0=b_2b_1$ and
$3b_2b_0=b_1^2$. The solutions of this system of equations are
$b_0=27\beta, b_1=27\beta\gamma, b_2=9\beta\gamma^2$ and $b_3 =
\beta\gamma^3$ for some $\beta,\gamma\in\QQ$. However, in this case
the polynomial $P_\xi(x)$ has a root $\frac{-3}{\gamma}$ which
contradicts the irreducibility of $P_\xi$.

In a similar way, we compute the leading coefficient
of~\eqref{eq29}:
$$
b_3(vwP'-3w^2P)^3 +
b_2w^2P'(vwP'-3w^2P)^2+b_1w^4(P')^2(vwP'-3w^2P)+b_0w^6(P')^3
$$$$
= (w^6P)(P')^3 - (3w^6P')P(P')^2 + 9(3b_3v+b_2w)w^5P'P^2 - 27w^6b_3P^3
$$
\begin{equation}\label{eq30}
\begin{array}{l}
= w^3P ((-18b_3b_2b_1+2b_2^3+27b_3^2b_0)v^3 +
(3b_2^2b_1-18b_3b_1^2+27b_3b_2b_0)v^2w\\[1ex]
\hspace{20ex}+ (18b_2^2b_0-3b_2b_1^2-27b_3b_1b_0)vw^2 +
(9b_2b_1b_0-2b_1^3-27b_3b_0^2)w^3).
\end{array}
\end{equation}
Therefore the leading coefficient can be written as $w^6
P(v/w)Q(v/w)$. The upshot is that the minimal polynomial of $\eta$
is $w^3Q(v/w)y^3 + 3w^2R(v/w)y^2 + 1$. Finally, we make the change
of variables $y\mapsto -w^3Q(v/w) y$ and the minimal polynomial of
the new value $\eta$ is
\begin{equation}\label{eq31}
y^3 - 3w^2R(v/w)y^2 - (w^3Q(v/w))^2.
\end{equation}
It is of the form~\textnumero 4 where $t = 3w^2R(v/w)$ and $a =
(w^3Q(v/w))^2$.

To find the roots of $Q$, we observe that the leading term of
$P_\eta$ equals zero if $\frac{u}{s} = \xi$ or
$$
\frac{-b_2z^2 - 2b_1z-3b_0}{3b_3z^2+2b_2z+b_1} = \xi,\qquad z =
\frac{v}{w}.
$$
This equation is quadratic in $z$. It is easy to verify that $z=\xi$
is its solution, which corresponds to the factor $w^3P$
in~\eqref{eq30}. Then the other root is
\begin{equation}\label{zx}
z(\xi) = -\frac{2(b_2\xi+b_1)}{3b_3\xi+b_2} - \xi.
\end{equation}
In particular, this means that the number of real roots of the
polynomial $Q$ coincides with the number of them for $P_\xi$, and
$z(x)$ is a bijection between the roots of $P$ and $Q$. Also notice
that $z(\xi)\neq \xi$ because otherwise the degree of $x$ is at most
2.

\begin{lemma}\label{lem11}
Let $a,t\in\CC$ be such that $|t|>2|a|^{1/3}$. Then the equation
\begin{equation}\label{lem11_eq}
x^3 - tx^2-a=0
\end{equation}
has exactly one root $\xi$ that satisfies $|\xi| > |a|^{1/3}$ and
the continued fraction~\textnumero 4 corresponds to $\xi$.
\end{lemma}

\proof Let $a$ be fixed. Since the leading term of the Laurent
series $x(t)$ from~\eqref{no4} is $t$, as soon as $x(t)$ converges,
the limit tends to infinity as $t\to\infty$. Moreover, for $t>T_0$,
where $T_0$ is the radius of convergence of $x(t)$, this function is
analytic and therefore continuous. Hence, $x(t)$ corresponds to the
root of the equation in $x^3 - tx^2-a=0$ which continuously tends to
infinity as $t\to\infty$.

Let $x$ be any complex number that satisfies $|x| = |a|^{1/3}$. Then
we have $|x^3 - a| \le 2|a| < |tx^2|$. Therefore the number of roots
of~\eqref{lem11_eq} inside $|x|< |a|^{1/3}$ is the same as the
number of roots of  $tx^2=0$ in the same region, i.e. it is two. We
thus get that there is only one root $\xi$ of~\eqref{lem11_eq}
outside that circle. Then, as $t\to\infty$, this root $\xi(t)$
continuously tends to infinity, because the other two roots can not
leave the circle $|x|<|a|^{1/3}$ and the sum of three roots
is~$t$.\endproof

\begin{lemma}\label{lem12}
Let $\xi$ be a root of the polynomial $P(x)$ and $z = z(\xi)$ given
by~\eqref{zx}. Define $d := |z-\xi|$. Suppose that integers $v$ and
$w$ are such that $3|R(v/w)|> 2|Q(v/w)|^{2/3}$, $\left|\frac{v}{w} -
z\right|<\frac{d}{2}$ and $\frac{d}{2 |s\xi-u|}>|w|\cdot
|Q(v/w)|^{2/3}$ where $u$ and $s$ are given by~\eqref{eq28}. Then
the continued fraction of~\eqref{eq31} corresponds to the root
$\eta$ such that $\xi = \frac{u\eta+v}{s\eta+w}$.
\end{lemma}

\proof Solving $\xi = \frac{u\eta+v}{s\eta+w}$ for $\eta$ gives
$$
|\eta| = \left|\frac{v-w\xi}{s\xi-u}\right| > \frac{dw}{2|s\xi-u|} >
\left|w^3Q\left(\frac{v}{w}\right)\right|^{2/3}.
$$
The second inequality is due to the fact that $\big|
\frac{v}{w}-\xi\big| \ge |z-\xi| - \big|z - \frac{v}{w}\big| >
\frac{d}{2}$. In view of~\eqref{eq31}, the right hand side of this
inequality is $|a|^{1/3}$. We also verify that the coefficients $t$
and $a$ satisfy $|t|>2|a|^{1/3}$: $3|w^2R(v/w)| >
2|w^3Q(v/w)|^{2/3}$. Therefore, all the conditions of
Lemma~\ref{lem11} are satisfied and $\eta$ corresponds to the
continued fraction~\eqref{no4} of~\eqref{eq31}.
\endproof

The last lemma shows that, as soon as we can find integer parameters
$v$ and $w$ that satisfy the lemma's conditions, we can provide a
continued fraction for $\eta = \frac{v-w\xi}{s\xi-u}$ where $\xi$ is
any given root of the cubic polynomial $P$. We still need to verify
that there exist $v$ and $w$ such that the continued
fraction~\eqref{no4} which corresponds to $\eta$, uniformly
converges and hence equals $\eta$. In order to do that, we first
provide the conditions on $a$ and $t$ that guarantee that the
continued fraction~\eqref{no4} converge. Here we proceed in a
similar way as for the one~\eqref{no5} in Section~\ref{sec4}.

Denote
$$
S_k:= \left(\begin{array}{cc} p_{4k}&q_{4k}\\
p_{4k-1}&q_{4k-1}
\end{array}\right);\quad T_k:= \left(\begin{array}{cc} p_{4k}&q_{4k}\\
p_{4k-4}&q_{4k-4}
\end{array}\right);\quad U_k:= \left(\begin{array}{cc} p_{4k-3}&q_{4k-3}\\
p_{4k-4}&q_{4k-4}
\end{array}
\right)
$$
From~\eqref{no4} and the recurrent formulae for convergents we get that
$S_{k+1} = A_kS_k$ where {\footnotesize
\begin{equation}
\begin{array}{rl}
A_k = &
\left(\begin{array}{cc} \!\!(8k+9)t&\hspace{-1ex} 3(12k+11)(6k+7)a\!\!\\
1&0
\end{array}
\right)
\left(\begin{array}{cc} \!\!2(8k+7)t^2&\hspace{-1ex}3(12k+7)(6k+5)a\!\!\\
1&0
\end{array}
\right)\\
&
\left(\begin{array}{cc} \!\!(8k+5)t&\hspace{-1ex}3(12k+5)(3k+2)a\!\!\\
1&0
\end{array}
\right)
\left(\begin{array}{cc} \!\!(8k+3)t^2&\hspace{-1ex} 3(12k+1)(3k+1)a\!\!\\
1&0
\end{array}
\right)
\end{array}
\end{equation}}
We also denote the four matrices involved in the product above by $C_{4k+1},
C_{4k+2}, C_{4k+3}$ and $C_{4k+4}$ so that $A_k =
C_{4k+4}C_{4k+3}C_{4k+2}C_{4k+1}$. Then the relation between $T_{k+1}$ and
$S_k$ is
\begin{equation}\label{eq33}
T_{k+1} = \left(\begin{array}{cc}
a_{k11}&a_{k12}\\
1&0
\end{array}
\right) S_k
\end{equation}
where $a_{k11}$ and $a_{k12} $ are the corresponding entries of the
matrix $A_k$. One then computes
\begin{equation}\label{eq43}
\begin{array}{rl}
a_{k11} = 2(8k+3)(8k+5)(8k+7)(8k+9)t^6 &+ 18(8k+5)(8k+7)(36k^2+55k+16)at^3\\
&+ 9(12k+5)(12k+11)(3k+2)(6k+7)a^2
\end{array};
\end{equation}$$
a_{k12} = 6(12k+1)(3k+1)(8k+7)at((8k+5)(8k+9)t^3 + 6(36k^2+63k+25)a).
$$

Next, we relate $U_k$ and $S_k$: $S_k = C_{4k}C_{4k-1}C_{4k-2}U_k$ or
$$
U_k = C^{-1}_{4k-2}C^{-1}_{4k-1}C^{-1}_{4k}S_k = \frac{B_k}{d}S_k,
$$
where, by~\eqref{no4}, $d = -27(12k-7)(12k-5)(12k-1)(3k-1)(6k-1)(6k+1)a^3$ and
$$
B_k = \left(\begin{array}{cc} 0&\hspace{-1.5ex}-3(12k-7)(3k-1)a\\
-1&(8k-3)t
\end{array}
\right) \left(\begin{array}{cc} 0&\hspace{-1.5ex}-3(12k-5)(6k-1)a\\
-1&2(8k-1)t^2
\end{array}\right) \left(\begin{array}{cc} 0&\hspace{-1.5ex}-3(12k-1)(6k+1)a\\
-1&(8k+1)t
\end{array}
\right).
$$
Then the relation between $T_k$ and $S_k$ is
\begin{equation}\label{eq32}
T_k = \left(\begin{array}{cc} 1&0\\
b_{k21}/d&b_{k22}/d
\end{array}
\right) S_k.
\end{equation}
where $b_{k21}$ and $b_{k22} $ are the corresponding components of
the matrix $B_k$. We compute
$$
b_{k21} = -3(12k-5)(6k-1)a - 2(8k-3)(8k-1)t^3,
$$$$
b_{k22} = 2(8k-1)t((8k-3)(8k+1)t^3+6(36k^2-9k-2)a) =: 2(8k-1)tp(k).
$$
Finally, we use~\eqref{eq32} and~\eqref{eq33} to relate $T_k$ and $T_{k+1}$:
\begin{equation}\label{eq34}
T_{k+1} = \left(\begin{array}{cc}
a_{k11} - a_{k12} \frac{b_{k21}}{b_{k22}}&\frac{da_{k12}}{b_{k22}}\\
1&0
\end{array}
\right) T_k.
\end{equation}
We then compute
{\footnotesize
\begin{equation}\label{eq37}
\frac{da_{k12}}{b_{k22}} = -\frac{81(12k+1)(12k-1)(12k-5)(12k-7)(3k-1)(3k+1)(6k-1)(6k+1)(8k+7)a^4p(k+1)}{(8k-1)p(k)}.
\end{equation}}

\begin{lemma}\label{lem13}
Let $a\in\ZZ$ be positive and $t$ satisfy $12a\le |t|^3$. Then
$q_{4k+4}$ and $q_{4k}$ satisfy the relation
$$
q_{4k+4} > (8k+3)(8k+5)(8k+7)(8k+9)(t^3+2a)^2 q_{4k}.
$$
\end{lemma}

If $t$ is positive then~\eqref{eq33} together with~\eqref{eq43} and
$12a\le t^3$ imply
$$
q_{4k+4}>a_{k11}q_{4k} > (8k+3)(8k+5)(8k+7)(8k+9)(t^3+2a)^2 q_{4k}.
$$
Assume now that $t$ is negative. In that case, we compare the terms
of $a_{k11}$ and use $12a\le |t|^3$ to get for $k\ge 1$
$$
18(8k+5)(8k+7)(36k^2+55k+16)a|t|^3 < (8k+3)(8k+5)(8k+7)(8k+9)t^6.
$$
Therefore, $a_{k11} > (8k+3)(8k+5)(8k+7)(8k+9)t^6$. On top of that,
by comparing the terms in $a_{k12}$, one can verify that
$(8k+5)(8k+9)|t|^3 > 6(36k^2+63k+25)a$ and therefore $a_{k12}>0$. We
combine the last two inequalities for $a_{k11}$ and $a_{k12}$ and
get
$$
q_{4k+4} = a_{k11}q_{4k} + a_{k12}q_{4k-4} >
(8k+3)(8k+5)(8k+7)(8k+9)t^6 q_{4k}.
$$
The lemma then follows from the fact that $|t^3| > |t^3+2a|$.
\endproof

\begin{lemma}\label{lem14}
Let $a\in\ZZ$ and $t\in\RR$ be such that $0<12a\le |t|^3$. Then the
continued fraction~\eqref{no4} converges. Moreover, there exists
$k_0>0$ such that for all $k>k_0$ its limit $\xi = x(t,a)$ satisfies
$$
\left|\frac{p_{4k}}{q_{4k}} - \frac{p_{4k+4}}{q_{4k+4}}\right| <
\left|\xi - \frac{p_{4k}}{q_{4k}}\right| < 2
\left|\frac{p_{4k}}{q_{4k}} - \frac{p_{4k+4}}{q_{4k+4}}\right|.
$$
\end{lemma}

\proof Equation~\eqref{eq34} infers that
\begin{equation}\label{eq36}
\frac{p_{4k+4}}{q_{4k+4}} - \frac{p_{4k}}{q_{4k}} = \frac{-da_{k12}q_{4k}q_{4k-4}}{b_{k22}q_{4k+4}q_{4k}}\left(\frac{p_{4k}}{q_{4k}} - \frac{p_{4k-4}}{q_{4k-4}}\right).
\end{equation}
In particular, since $d$ is negative and $a_{k12}$ and $b_{k22}$ are
positive numbers, as was shown in the proof of the previous lemma,
the differences $\frac{p_{4k+4}}{q_{4k+4}}-\frac{p_{4k}}{q_{4k}}$
share the same sign for all $k\ge 1$. Now we use Lemma~\ref{lem13}
and~\eqref{eq37} to estimate{\footnotesize$$
\frac{-da_{k11}q_{4k-4}}{b_{k22}q_{4k+4}} < \frac{
81(12k-1)(12k+1)(12k-5)(12k-7)(9k^2-1)(36k^2-1)a^4}{(8k-5)(8k-3)(8k-1)(8k+1)(8k+3)(8k+5)(8k+7)(8k+9)(t^3+2a)^4}
\cdot \frac{(8k+7)p(k+1)}{(8k-1)p(k)}.
$$}
Since $p(k)$ is a quadratic polynomial in $k$, the second fraction
in the last expression tends to one as $k\to\infty$. The first
fraction tends to $\frac{3^{12}2^{10}}{2^{24}}\cdot
\big(\frac{a}{t^3+2a}\big)^4$, which is strictly less than $\frac12$
as soon as $5a<|t^3+2a|$.

For the rest of the proof we proceed analogously to the proof of
Lemma~\ref{lem7}. We write $\frac{p_{4k+4}}{q_{4k+4}}$ as the sum
$$
\frac{p_{4k+4}}{q_{4k+4}} = \frac{p_0}{q_0} + \sum_{i=0}^k
\left(\frac{p_{4i+4}}{q_{4i+4}} - \frac{p_{4i}}{q_{4i}}\right)
$$
Since, starting from some $k\ge k_0$ the terms in the summation on
the right hand side decay by a factor at least 2, the limit of
$\frac{p_{4k+4}}{q_{4k+4}}$ as $k\to\infty$ exists, let's call it
$\xi$. Then the estimates on $\big|\xi - \frac{p_{4k}}{q_{4k}}\big|$
follow easily.
\endproof

Now we are ready to complete the proof of Theorem~\ref{th2}. Let
$\xi$ be a root of $P$. We need to construct the integers $v$ and
$w$ such that the resulting coefficients $a$ and $t$ in~\eqref{eq31}
satisfy all the conditions of Lemmata~\ref{lem13} and~\ref{lem14}.
Denote by $B$ the height of $\xi$, i.e. $B:=\max\{|b_0|, |b_1|,
|b_2|, |b_3|\}$.

We construct $v,w$ such that $v/w$ is very close to $z(\xi)$ which
is defined in~\eqref{zx}. By the Dirichlet theorem, there exist
$v,w\in\ZZ$ such that $w$ is arbitrarily large and
$$
\left|z - \frac{v}{w}\right|< \frac{1}{w^2}.
$$
Since $z$ is a root of the polynomial $Q$ from~\eqref{eq31}, we get
for $z$ close enough to $v/w$,
$$
|w^3Q(v/w)|\le w^3 \cdot 2\left|z-\frac{v}{w}\right|\cdot |Q'(z)|\le 2|Q'(z)|w.
$$

Now we estimate $\big|\frac{u}{s} - \xi\big|$ where $u$ and $s$ are
defined in~\eqref{eq28}.
$$
\left|\frac{u}{s} - \xi\right| = |D(v/w) - D(z)|,
$$
where $D(z)$ is the rational function $D(z):=
\frac{-b_2z^2-2b_1z-3b_0}{3b_3z^2+2b_2z+b_1}$. Since $z$ is cubic
irrational, it is not a singularity of $D$. We choose $v/w$ close
enough to $z$ so that the whole interval between $z$ and $v/w$ does
not contain singularities of $D$. Moreover, one can check hat
$D'(z)\neq0$, thus for $z$ close enough to $v/w$,
\begin{equation}\label{eq38}
|D(v/w) - D(z)| < 2\left|\frac{v}{w}-z\right| |D'(z)| < \frac{2|D'(z)|}{w^2}.
\end{equation}
Since $|s| = |w^2 P'(v/w)|\le 6Bw^2|z|^2$, we get that
$$
|u - s\xi| \le 12B|z^2\cdot D'(z)|
$$
and therefore for $w$ large enough, we get
$$
w|Q(v/w)|^{2/3} \le \frac{|2Q'(z)|^{2/3}}{w^{1/3}} \le
\frac{d}{24B|z^2\cdot D'(z)|} \le \frac{d}{2|u-s\xi|},
$$
where $d = |z-\xi|$.

Consider the polynomial $R$ from~\eqref{eq39}. Since it is not the
constant zero and has degree at most 2, $R(z)\neq 0$. Therefore, for
$v/w$ close enough to $z$ we have $|R(v/w)|> \frac12 |R(z)|>0$.
Therefore
$$
|3R(v/w)| > \frac32 |R(z)|.
$$
On the other hand,
$$
|2Q(v/w)|^{2/3} < \left|\frac{4Q'(z)}{w^2}\right|^{2/3}.
$$
Therefore, for large enough $w$ we get that $|3R(v/w)| >
|2Q(v/w)|^{2/3}$.

The last condition to check is one in Lemma~\ref{lem14} which is
written as $|12(w^3 Q(v/w))^2|<|3w^2R(v/w)|^3$. This is indeed true
because, for $w$ large enough,
$$
|12(w^3Q(v/w))^2| < 12\cdot 4(Q'(z))^2 w^2 < \frac{27}{8}|R(z)|^3
w^6 < |3w^2R(v/w)|^3.
$$

All the conditions of Lemmata~\ref{lem13} and~\ref{lem14} are
satisfied and therefore $\eta$, which solves the equation $\xi =
\frac{u\eta+v}{s\eta+w}$ admits a continued fraction expansion of
the form~\eqref{no4}. Moreover, that continued fraction converges to
$\eta$. This finished the proof of Theorem~\ref{th2}.
\appendix

\section{Formulae for coefficients of the remaining continued
fractions}\label{app1}

In Section~\ref{sec1} we only verified the continued fractions of
the cubic Laurent series~\textnumero 1 and~\textnumero 5. As
discussed in the Introduction, the continued fraction for the
series~\textnumero 2 follows from the one for~\textnumero 1. The
verification of the remaining continued fractions repeats all the
steps in Subsection~\ref{subsec32}. Therefore, here we only provide
the formulae for the coefficients $A_i, B_i, C_i$ and $D$ as well as
$a_i, \beta_i$. An interested reader can then verify them by
induction in the same way as for the algebraic series~\textnumero 5.

{\bf Algebraic series \textnumero 3: $x^3-tx^2-at=0$.}

$$
D = 4t^3 + 27at;
$$

{
$$
A_{4k+1} = -4t,\; B_{4k+1} = -(32k+8)t^2 - 9a,\; C_{4k+1} =
12(12k+1)(3k+1)at,
$$
$$
 a_{4k+1} = (8k+3)t,\;
\beta_{4k+1} = 3(12k+1)(3k+1)a;
$$
$$
A_{4k+2} = -4t,\; B_{4k+2} = -(32k+16)t^2 + 9a,\; C_{4k+2} =
12(12k+5)(3k+2)at;
$$
$$
a_{4k+2} = (8k+5)t,\; \beta_{4k+2} = 3(12k+5)(3k+2)a;
$$
$$
A_{4k+3} = -4t,\; B_{4k+3} = -(32k+24)t^2 - 9a,\; C_{4k+3} =
6(12k+7)(6k+5)a;
$$
$$
a_{4k+3} = 2(8k+7)t,\; \beta_{4k+3} = 3(12k+7)(6k+5)a;
$$
$$
A_{4k+4} = -2t,\; B_{4k+4} = -(32k+32)t^2 + 9a,\; C_{4k+4} =
12(12k+11)(6k+7)t;
$$
$$
a_{4k+4} = (8k+9)t,\; \beta_{4k+4} = 3(12k+11)(6k+7)a.
$$}

{\bf Algebraic series~\textnumero 4: $x^3-tx^2-a=0$.} As discussed
in the Introduction, the continued fraction for this series is
derived from that for the series~\textnumero 3.

{\bf Algebraic series \textnumero 6: $x^3+(t-2)x^2-2(t-2)x +
2(t-2)=0$.}
$$
D = (t-2)(t^2+6t+11);
$$
{\footnotesize
$$
A_{3k+1} = (-1)^{3k}((8k+1)t+(20k+1)),\; B_{3k+1} = -2(4k+1)^2t^2-
(96k^2+48k+5)t-(112k^2+56k+3),
$$$$
C_{3k+1} = (-1)^{3k}\cdot2(6k+1)(3k+1)((8k+3)t+(20k+9)),\; D_{3k+1}
= (4k+1)D;
$$
$$
 a_{3k+1} = (-1)^{3k}((4k+1)t+2k),\;
\beta_{3k+1} = 2(6k+1)(3k+1);
$$
$$
A_{3k+2} = (-1)^{3k+1}((8k+3)t+(20k+9)),\; B_{3k+2} =
-(4k+1)((8k+4)t^2+(24k+13)t-(8k+3)),
$$$$
C_{3k+2} = (-1)^{3k+1}\cdot 12(4k+1)^2(3k+2),\; D_{3k+2} = (4k+1)D;
$$
$$
a_{3k+2} = (-1)^{3k+1}(4k+3)(t^2+3t-1),\; \beta_{3k+2} =
6(4k+1)(3k+2);
$$
$$
A_{3k+3} = (-1)^{3k}\cdot 2,\; B_{3k+3} = -(8k+8)t^2 - (24k+23)t +
(8k+9),
$$$$
C_{3k+3} = (-1)^{3k}(18k+15)((8k+9)t+(20k+21)),\; D_{3k+3} = D;
$$
$$
a_{3k+3} = (-1)^{3k}((4k+5)t+(2k+3)),\; \beta_{3k+3} =
3(4k+5)(6k+5).
$$
}

\section{Continued fractions of cubic irrationals with very large partial
quotients}\label{app2}

Here we present the list of cubic numbers $\xi\in\RR$ whose
continued fractions have at least one partial quotient $a_n$ such
that
\begin{equation}\label{app2_eq}
a_n\ge 2H(\xi)^{\tau_1} n^2,\qquad \tau_1 = 3 + \frac{2\ln 2}{2.88}
\approx 3.4814...
\end{equation}
If two or more equivalent numbers satisfy~\eqref{app2_eq}, we
present only one of them with the largest value of $C:=
\frac{a_n}{H(\xi)^{\tau_1} n^2}$. Here, we say that $\xi$ and
$\zeta$ are equivalent if their continued fractions eventually
coincide.

This list is produced by a computer search of roots of cubic
polynomials with small height. To simplify the algorithm, we only go
through polynomials that take opposite signs at 0 and 1 and only
consider their largest real root in the interval $(0,1)$. The code
was implemented in C++ computer language with NTL library for long
arithmetic operations. It can be provided by request.

Under the above constraints, for all algebraic $\xi$ with $H(\xi)\le
2$ the first 10000 partial quotients were computed; for all $\xi$
with $H(\xi)\le 5$ the first $5000$ partial quotients; for all $\xi$
with $H(\xi)\le 10$ the first $1000$ partial quotients and for all
$\xi$ with $H(\xi)\le 100$ the first $50$ partial quotients were
computed. These calculations took around 19 hours 49 minutes on one
core of Ryzen 3700X CPU.

The resulting list is:
\begin{enumerate}
\item Root of $x^3+x^2+x-1=0$ has $\xi = [0;1,1,5,4,2,{\bf
305},\ldots]$. For $a_6(\xi)$ the value $\displaystyle C:=
\frac{a_6(x)}{6^2\cdot H(\xi)^{\tau_1}}$ equals $8.472\ldots$
\item Root of $2x^3+2x-1=0$ has $\xi = [0; 2, 2, 1, 3, 1, 1, 1, 1, 2, 1, 5, 456, 1, 30, 1, 3, 4,
{\bf 29866}, \ldots]$. For $a_{18}(\xi)$, $C = 8.253\ldots$
\item Root of $2x^3+2x^2+2x-1=0$ has $$\xi = [0; 2, 1, 11, 2, 3, 1, 23, 2, 3, 1, 1337, 2, 8, 3, 2, 1, 7, 4, 2, 2,
{\bf 87431},\ldots].$$ For $a_{21}(\xi)$, $C = 17.751\ldots$
\item Root of $x^3-2x^2-3x+1=0$ has $$\xi = [0; 3, 2, 26, 1, 6, 3, 3, 1, 2, 4, 92, 24, 2, 3, 2, 4, 2, 1, 16,
{\bf 40033},\ldots].$$ For $a_{20}(\xi)$, $C = 2.184\ldots$.
\item Root of $2x^3-2x^2+4x-3=0$ has
$$
\hspace{-6ex}\begin{array}{rl} \xi = &[0; 1, 4, 3, 7, 4, 2, 30, 1,
8, 3, 1, 1, 1, 9, 2, 2, 1, 3,
22986, 2, 1, 32, 8, 2, 1, 8, 55, 1, 5, 2, 28,\\
& 1, 5, 1, {\bf 1501790}, 1, 2, 1, 7, 6, 1, 1, 5, 2, 1, 6, 2, 2, 1,
2, 1,
1, 3, 1, 3, 1, 2, 4, 3, 1, 35657, 1,\\
& 17, 2, 15, 1, 1, 2, 1, 1, 5, 3, 2, 1, 1, 7, 2, 1, 7, 1, 3, 25,
49405, 1, 1, 3, 1, 1, 4, 1, 2, 15, 1, 2, 83,\\
& 1, 162, 2, 1, 1, 1, 2, 2, 1, 53460, 1, 6, 4, 3, 4, 13, 5, 15, 6,
1, 4, 1, 4, 1, 1, 2, 1, {\bf 16467250},\ldots].
\end{array}
$$
For $a_{35}(\xi)$, $C = 9.828\ldots$; for $a_{123}(\xi)$, $C =
8.726\ldots$

\item Root of $7x^3+4x^2-4x-6=0$ has
$$
\hspace{-6ex}\begin{array}{rl} \xi = &[0; 1, 22, 1, 31, 2, 3, 1, 63,
1, 10, 1, 2, 1, 7, 1, 160905, 2, 1, 4, 58, 2, 2, 1, 2, 1, 7, 3, 1,
3, 1, 4,\\
& 3, 1, 47, 1, 214540, 1, 2, 9, 1, 45, 1, 3, 1, 48, 1, 21, 1, 9, 1,
8, 1, 2, 249610, 1, 1, 1, 1, 1, 3, 1,\\
& 1, 1, 1, 20, 1, 4, 19, 1, 2, 1, 1, 1, 1, 3, 4, 1, 1, 1, 1, 3,
345838, 1, 13, 1, 3, 3, 1, 1, 1, 1, 9, 1, 11,\\
& 7, 23, 5, 13, 1, 374230, 31, 6, 2, 1, 2, 5, 3, 1, 1, 7, 4, 1, 37,
{\bf 115270760},\ldots].
\end{array}
$$
For $a_{114}(\xi)$, $C = 10.134\ldots$
\item Root of $14x^3+10x^2+8x-5=0$ has
$$
\hspace{-6ex}\begin{array}{rl} \xi = &[0; 2, 1, 2, 1, 1, 11, 1, 1,
4, 1, 1, 1, 10, 24, 6, 2, 8, 436745, 2, 1, 1, 16, 1, 29, 2, 1, 2, 2,
1,
1, 3,\\
& 34, 2, 1, 3, {\bf 28534040},\ldots].
\end{array}
$$
For $a_{36}(\xi)$, $C = 2.252\ldots$
\item Root of $11x^3+21x^2+24x-30=0$ has
$$
\hspace{-6ex}\begin{array}{rl} \xi = &[0; 1, 2, 5, 1095, 2, 1, 2, 5,
1, 8, 2, 5, 1, 14, 1, 1, 1, 2, 11, 11, 1, 2, 4, 1, 1, 1, 1, 1, 9, 1,
6, 2,\\
& {\bf 829131361},\ldots].
\end{array}
$$
For $a_{33}(\xi)$, $C = 5.485\ldots$
\item Root of $44x^3+42x^2+24x-15=0$ has
$$
\hspace{-6ex}\begin{array}{rl} \xi = &[0; 2, 1, 10, 547, 1, 2, 5, 2,
1, 17, 1, 11, 1, 6, 1, 4, 2, 1, 5, 23, 2, 1, 1, 1, 4, 3, 4, 1, 13,
1,\\
& {\bf 1658262722},\ldots].
\end{array}
$$
For $a_{31}(\xi)$, $C = 3.277\ldots$.
\end{enumerate}

\bigskip
\noindent Dzmitry Badziahin\\ \noindent The University of Sydney\\
\noindent Camperdown 2006, NSW (Australia)\\
\noindent {\tt dzmitry.badziahin@sydney.edu.au}

\end{document}